\documentclass[10pt,reqno]{amsart}

\usepackage{setspace}
\setdisplayskipstretch{1.5}

\usepackage[T1]{fontenc}
\usepackage{lmodern}

\usepackage{mathtools}
\usepackage{amsmath,amsfonts,amsbsy,amsgen,amscd,mathrsfs,amssymb,amsthm}
\usepackage{enumitem}
\usepackage{bm}
\usepackage{calc}

\usepackage{stmaryrd}
\SetSymbolFont{stmry}{bold}{U}{stmry}{m}{n} 

\usepackage[usenames,dvipsnames]{xcolor}
\usepackage[colorlinks=true,citecolor=blue,linkcolor=blue]{hyperref}

\usepackage{tikz}
\usepackage[font=small]{caption}

\newtheorem{thm}{Theorem}[section]
\newtheorem{lem}[thm]{Lemma}

\newtheorem{cor}[thm]{Corollary}

\theoremstyle{definition}

\newtheorem{defn}[thm]{Definition}

\theoremstyle{remark}

\newtheorem{example}[thm]{Example}

\newtheorem{rem}[thm]{Remark}
\newtheorem*{rem*}{Remark}

\numberwithin{equation}{section} 
\numberwithin{figure}{section}
\numberwithin{table}{section}

\makeatletter

\newcommand{\Opr}{O_\mathrm{P}}
\newcommand{\supp}{\mathop{\mathrm{supp}}}
\newcommand{\tr}{\mathop{\mathrm{Tr}}}
\newcommand{\ntrN}{\mathop{\mathrm{tr}_N}}
\newcommand{\ntrNalp}{\mathop{\mathrm{tr}_{N^\alpha}}}
\newcommand{\ntrD}{\mathop{\mathrm{tr}_D}}
\newcommand{\ntrND}{\mathop{\mathrm{tr}_{DN}}}
\newcommand{\ntrDDN}{\mathop{\mathrm{tr}_{DD_N}}}
\newcommand{\ntrDN}{\mathop{\mathrm{tr}_{D_N}}}

\newcommand{\id}{\mathbf{1}}
\newcommand{\EE}{\mathbf{E}}

\begin{document}

\title{\vspace*{-1cm}
A new approach to strong convergence}

\author[Chen]{Chi-Fang Chen}
\address{EECS, University of California, Berkeley, CA 94720, USA
\newline\indent
Center for Theoretical Physics, Massachusetts Institute of 
Technology, Cambridge, MA 02139, USA}
\email{achifchen@gmail.com}

\author[Garza-Vargas]{Jorge Garza-Vargas}
\address{Department of Computing and Mathematical Sciences, 
California Institute of Technology, Pasadena, CA, USA}
\email{jgarzav@caltech.edu}

\author[Tropp]{Joel A. Tropp}
\address{Department of Computing and Mathematical Sciences, 
California Institute of Technology, Pasadena, CA, USA}
\email{jtropp@caltech.edu}

\author[Van Handel]{Ramon van Handel}
\address{Department of Mathematics, Princeton University, 
Princeton, NJ 08544, USA}
\email{rvan@math.princeton.edu}

\begin{abstract}
A family of random matrices $\boldsymbol{X}^N=(X_1^N,\ldots,X_d^N)$ 
is said to converge strongly to a family of bounded operators 
$\boldsymbol{x}=(x_1,\ldots,x_d)$ when
$\|P(\boldsymbol{X}^N,\boldsymbol{X}^{N*})\|\to\|P(\boldsymbol{x}, 
\boldsymbol{x}^*)\|$ for every noncommutative polynomial $P$. This 
phenomenon plays a key role in several recent breakthroughs on 
random graphs, geometry, and operator algebras. 
However, proofs of strong convergence are notoriously delicate and have
relied largely on problem-specific methods.

In this paper, we develop a new approach to strong convergence that uses 
only soft arguments. Our method exploits the fact that for many natural 
models, the expected trace of $P(\boldsymbol{X}^N,\boldsymbol{X}^{N*})$ 
is a rational function 
of $\frac{1}{N}$ whose lowest order asymptotics are easily understood. We 
develop a general technique to deduce strong convergence directly from 
these inputs using the inequality of A.\ and V.~Markov for univariate 
polynomials and elementary Fourier analysis.
To illustrate the method, we develop the following applications.
\vspace*{.5mm}
\begin{enumerate}[leftmargin=*,label=\arabic*.]
\addtolength{\itemsep}{.5mm}
\item We give a short proof of the result of 
Friedman that random regular graphs have a near-optimal spectral gap,
and obtain a sharp understanding of the large deviations 
probabilities of the second eigenvalue.
\item We prove a strong quantitative form of the strong convergence
property of random permutation matrices due to Bordenave and Collins.
\item We extend the above to any
stable representation
of the symmetric group, providing many new examples of the strong 
convergence phenomenon.
\end{enumerate}
\end{abstract}



\maketitle

\section{Introduction}
\label{sec:intro}

Let $\boldsymbol{X}^N=(X_1^N,\ldots,X_d^N)$ be a sequence of $d$-tuples of 
random matrices of increasing dimension, and let
$\boldsymbol{x}=(x_1,\ldots,x_d)$ be a $d$-tuple of bounded operators.
Then $\boldsymbol{X}^N$ is said to \emph{converge strongly} 
to $\boldsymbol{x}$ if
$$
	\|P(\boldsymbol{X}^N,\boldsymbol{X}^{N*})\|
	\xrightarrow{N\to\infty}
	\|P(\boldsymbol{x},\boldsymbol{x}^*)\|
	\quad\text{in probability}
$$
for every noncommutative polynomial $P$.

The notion of strong convergence has its origin in the work of 
Haagerup and Thorbj{\o}rnsen \cite{HT05} in the context of complex 
Gaussian (GUE) matrices. It has since proved to be a powerful tool in a 
broad range of applications. Notably, strong convergence plays a 
central role in a series of recent breakthroughs in the study of random 
lifts of graphs \cite{BC19}, random covering spaces 
\cite{HM23,LM23,Hid23,MT23}, random minimal surfaces \cite{Son24}, and 
operator algebras \cite{HT05,HST06,Hay22,BC22,HJE23,BC24}; we refer to the 
recent paper \cite{BC24} for a more extensive discussion and bibliography.

In many cases, norm convergence can be highly nontrivial to establish even 
for the simplest polynomials $P$. For example, let $\bar S_1^N,\ldots,\bar 
S_d^N$ be i.i.d.\ random permutation matrices of dimension $N$. Then the
linear polynomial
$$
	\bar A^N = \bar S_1^N + \bar S_1^{N*} + \cdots +
	\bar S_d^N + \bar S_d^{N*}
$$
is the adjacency matrix of a random $2d$-regular graph. Every $2d$-regular 
graph has largest eigenvalue $2d$ with eigenvector $1_N$, while the 
Alon--Boppana bound states that the second eigenvalue is at least 
$2\sqrt{2d-1}-o(1)$ \cite{Nil91}. It was conjectured by Alon, and proved 
in a breakthrough monograph of Friedman \cite{Fri08}, that the random 
$2d$-regular graph satisfies $\|\bar A^N|_{\{1_N\}^\perp}\| = 
2\sqrt{2d-1}+o(1)$, so that such graphs have near-optimal second 
eigenvalue. The much more general result that random permutation matrices 
converge strongly, due to Bordenave and Collins \cite{BC19}, is of major 
importance in many recent applications cited above.

Given the power of the strong convergence phenomenon, it may be 
unsurprising that strong convergence has been notoriously difficult to 
prove (see, e.g., the discussion in \cite{BC20}).
In those cases where strong convergence has been established, the 
proofs rely on problem-specific methods and often require delicate 
computations. These obstacles have hampered efforts to establish strong 
convergence in new situations and to obtain a sharper understanding of 
this phenomenon.

In this paper, we develop a new approach to strong convergence that uses 
only soft arguments, and which requires limited problem-specific 
inputs as compared to previous methods. We will illustrate the 
method in 
the context of 
random permutation matrices and other representations of the symmetric 
group, resulting in surprisingly short proofs of strong convergence and 
new quantitative and qualitative information on the strong convergence 
phenomenon:
\vspace*{.5mm}
\begin{itemize}[leftmargin=*]
\addtolength{\itemsep}{.2cm}
\item We prove a quantitative form of Friedman's theorem on the
second eigenvalue of random regular graphs \cite{Fri08} that
reveals new probabilistic structure. In particular,
we obtain a
sharp characterization of the tail behavior
of the second eigenvalue. 
\item We obtain a quantitative form of the strong 
convergence of random permutation matrices due to Bordenave and Collins 
\cite{BC19,BC24}, which significantly improves the best known convergence 
rate for arbitrary polynomials.
\item We establish strong convergence of random matrices defined by any 
stable representation of the symmetric group (in the sense of \cite{Far14}).
This provides a large new family of examples of the strong convergence 
phenomenon, and illustrates that strong convergence arises in settings
far beyond the standard representation.
\end{itemize}
Further applications to several other models will appear in forthcoming 
work.

Let us emphasize at the outset that the main difficulty in establishing
strong convergence is to prove the upper bound 
$$
	\|P(\boldsymbol{X}^N,\boldsymbol{X}^{N*})\| \le
	\|P(\boldsymbol{x},\boldsymbol{x}^*)\| + o(1)
	\quad\text{with probability }1-o(1)
$$
as $N\to\infty$. The corresponding lower bound then follows 
automatically, either as a byproduct of the proof or from 
general principles (see, e.g., \cite[pp.\ 16--19]{LM23}).
We therefore focus on upper bounds throughout this 
paper, and relegate a brief treatment of the corresponding lower
bounds to Appendix \ref{sec:lowerbd}.

\subsection*{Organization of this paper}

This paper is organized as follows. In section~\ref{sec:outline}, we first 
outline the core ingredients of our approach in a general setting, while 
section~\ref{sec:main} presents the main results obtained in this paper.

The next two sections introduce some general tools on which our methods 
are based. Section~\ref{sec:basic} recalls properties of polynomials and 
distributions that form the basis for our approach, while 
section~\ref{sec:words} recalls basic properties of words in random 
permutation matrices that form the input for our methods.

The rest of the paper is devoted to the proofs of our main results. 
Sections~\ref{sec:masterI} and~\ref{sec:masterII} prove master 
inequalities for polynomials and smooth functions, respectively. These 
inequalities are applied in section~\ref{sec:rreg} to random regular 
graphs, and in section~\ref{sec:bc} to strong convergence of random 
permutation matrices. Section~\ref{sec:stable} extends the above 
results to general stable representations of the symmetric group.

Finally, Appendix~\ref{sec:lowerbd} is devoted to a brief treatment of
lower bounds.

\subsection*{Notation}

In the following, $a\lesssim b$ denotes that $a\le Cb$ for a universal 
constant $C$. We denote
$\mathbb{N}:=\{1,2,3,\ldots\}$ and $\mathbb{Z}_+:=\{0,1,2,\ldots\}$.

The symmetric group on $N$ letters is denoted as 
$\mathbf{S}_N$, the free group with free generators $g_1,\ldots,g_d$ 
is denoted $\mathbf{F}_d$, and $e\in\mathbf{F}_d$ denotes
the identity element.

We denote by $\mathcal{P}_q$ the space of real polynomials 
$h:\mathbb{R}\to\mathbb{R}$ of degree at most $q$, and by $\mathcal{P}$ 
the space of all real polynomials. We denote by $h^{(m)}$ the $m$th 
derivative of a univariate function, and by $\|h\|_{C^m[a,b]}:= 
\sum_{j=0}^m \sup_{x\in[a,b]} |h^{(j)}(x)|$.

We will denote by $\mathrm{M}_N(\mathbb{C})$ the space of $N\times N$ 
matrices with complex entries. The trace of a matrix $M$ is denoted as 
$\tr M$, the normalized trace as $\ntrN M := \frac{1}{N}\tr 
M$, and the operator norm as $\|M\|$. The identity matrix is denoted 
$\mathbf{1}_N\in 
\mathrm{M}_N(\mathbb{C})$, and $1_N\in\mathbb{C}^N$ denotes the vector all 
of whose entries equal one. We use the convention that powers bind
before the trace, e.g., $\ntrN (X^N)^p := \ntrN[(X^N)^p]$.

If $X$ is a self-adjoint operator and $h:\mathbb{R}\to\mathbb{C}$, then 
the operator $h(X)$ is defined by the usual functional calculus (in 
particular, if $X$ is a self-adjoint matrix, $h$ is applied to the 
eigenvalues while keeping the eigenvectors fixed).

\section{Outline}
\label{sec:outline}

The main results of this paper will be presented in section 
\ref{sec:main}, which can be read independently of the present section. 
However, as the new approach that underlies their proofs is much more 
broadly applicable, we begin in this section by introducing the core 
ingredients of the method in a general setting.

Throughout this section, we let $X^N$ be a sequence of $N$-dimensional 
self-adjoint random matrices, whose 
expected empirical eigenvalue distribution
converges as $N\to\infty$ to a limiting probability measure $\nu_0$ with 
compact support. This \emph{weak convergence} property is captured by 
convergence of the moments
$$
	\lim_{N\to\infty} \EE[\ntrN (X^N)^p] =
	\int x^p\, \nu_0(dx)\qquad\text{for all }p\in\mathbb{N}.
$$
The moments $\EE[\ntrN (X^N)^p]$ generally admit a combinatorial
description whose behavior as $N\to\infty$ is well understood,  
so that weak convergence is readily accessible. However, weak convergence 
can only ensure that a fraction $1-o(1)$ of eigenvalues converges to the 
support of $\nu_0$, and cannot rule out existence of outliers in the 
spectrum. This is the basic difficulty in strong convergence problems.

In the problems of interest in this paper, the random matrix 
$X^N=P(\boldsymbol{X}^N,\boldsymbol{X}^{N*})$ is a noncommutative 
polynomial of a given family $\boldsymbol{X}^N=(X_1^N,\ldots,X_d^N)$ of 
random matrices. In this setting, it is natural to think of
$\nu_0$ itself as the spectral distribution of a certain limiting operator
$X_{\rm F}=P(\boldsymbol{x},\boldsymbol{x}^*)$ in the sense that
$$
	\tau(X_{\rm F}^p) = \int x^p\, \nu_0(dx)\qquad\text{for all }
	p\in\mathbb{N},
$$
where $\boldsymbol{x}=(x_1,\ldots,x_d)$ is a limiting family of operators
associated to $\boldsymbol{X}^N$ and $\tau$ is a positive linear 
functional that plays the role of the trace for the limiting 
operators. Such models are naturally 
formulated in the setting of $C^*$-probability spaces, for which we refer 
to \cite{NS06} for a excellent introduction. However, this general framework 
will not be needed in this paper, as the limiting objects associated to 
the models we will study admit an explicit construction that is 
recalled in section \ref{sec:prelim}. 

\begin{example}
\label{ex:rgraph}
While the aim of this section in to discuss the methods of this paper in a 
general setting, the reader may keep in mind the following guiding 
example. Let $X^N=\bar A^N|_{\{1_N\}^\perp}$ be the adjacency matrix of
a random $2d$-regular graph with the trivial eigenvalue removed, as 
defined in the introduction. In this case, the limiting operator $X_{\rm F}$
may be viewed as the adjacency matrix of the infinite $2d$-regular tree, 
and $\tau(X)=\langle \delta_e,X\delta_e\rangle$ where $\delta_e$ is the 
coordinate vector associated with the root of the tree. 
The spectral distribution $\nu_0$ of $X_{\rm F}$ is the 
Kesten-McKay distribution, which is supported in the interval
$[-\|X_{\rm F}\|,\|X_{\rm F}\|]$ with
$\|X_{\rm F}\|=2\sqrt{2d-1}$. A more precise description of 
this model 
is given in section \ref{sec:rregmain} below.
\end{example}

\begin{rem}
\label{rem:sanoloss}
Throughout this paper, we will work with \emph{self-adjoint} random 
matrices $X^N$ and operators $X_{\rm F}$. This entails no loss of 
generality: since $\|X\|^2=\|X^*X\|$ for any bounded operator $X$, strong
convergence of arbitrary noncommutative polynomials $P$ is equivalent to
strong convergence of the self-adjoint polynomials $P^*P$. We may  
therefore restrict attention to self-adjoint polynomials $P$.

While it is not essential for the applicability of the methods of this 
paper, we will also assume for simplicity that the random matrices $X^N$ are 
uniformly bounded, i.e., there is a constant $K$ so that 
$\|X^N\|\le K$ a.s.\ for all $N$. This is the case for all models 
considered in this paper; for example, $\|X^N\|\le 2d$ a.s.\ 
in Example \ref{ex:rgraph}.
\end{rem}

In the remainder of this section, we assume that weak 
convergence
\begin{equation}
\label{eq:weakcv}
	\lim_{N\to\infty}
	\EE[\ntrN (X^N)^p] 
	= \tau(X_{\rm F}^p)
	\qquad\text{for all }p\in\mathbb{N}
\end{equation}
holds.
The problem of \emph{strong convergence} is to show that not only the 
moments, but also the operator norm, of $X^N$ 
converges to that of the limiting model $X_{\rm F}$, which ensures a lack 
of outliers in the spectrum. More precisely, we aim to explain in the 
sequel how to prove the strong convergence upper bound
$\|X^N\| \le \|X_{\rm F}\|+ o(1)$, which is the main difficulty.
The corresponding lower bound is typically an easy consequence of 
\eqref{eq:weakcv}, as will be discussed in Appendix \ref{sec:lowerbd}.

To date, proofs of strong convergence were based on resolvent 
equations \cite{HT05,Sch05}, interpolation \cite{CGP22,Par22,Par23,BBV23}, 
or the moment method \cite{BC19,BC20,BC24}. The first two approaches rely 
on analytic tools, such as integration by parts formulae or free 
stochastic calculus, that are only available for special models. In 
contrast, the only known way to access the properties of models such as 
those based on random permutations or group representations is through 
moment computations. We therefore begin by 
recalling the challenges faced by the moment method.

\subsection{Prelude: the moment method}
\label{sec:prelude}

The basic premise behind the moment method is the observation that since
$|\tau(X_{\rm F}^p)|\le \|X_{\rm F}\|^p$, 
\eqref{eq:weakcv} implies that
\begin{equation}
\label{eq:momentmethod}
	\big(\EE\|X^N\|^{2p}\big)^{1/2p} \le
	\big(\EE[\tr (X^N)^{2p}]\big)^{1/2p} \le
	N^{1/2p} \big(\|X_{\rm F}\|+o(1)\big)
\end{equation}
as $N\to\infty$. This does not in itself suffice to prove $\|X^N\|\le 
\|X_{\rm F}\|+o(1)$, as the right-hand side diverges when $p$ is fixed. 
The essence of the moment method is that the desired bound would follow if 
\eqref{eq:momentmethod} remains valid for $p=p(N)\gg\log N$, as then
$N^{1/2p} = 1+o(1)$. We emphasize this is a much stronger property 
than \eqref{eq:weakcv}.
In practice, the implementation of this method faces two major 
obstacles.\vspace*{.5mm}
\begin{itemize}[leftmargin=*]
\addtolength{\itemsep}{.2cm}
\item While moment asymptotics for fixed $p$ are 
accessible in many situations, the case where $p$ grows rapidly with $N$ 
is often a difficult combinatorial problem. Despite the availability of
tools to facilitate the analysis of large moments for
strong convergence, such as nonbacktracking and linearization
methods (see, e.g., \cite{BC19,BC24}), the core of the analysis remains 
dependent on delicate computations that do not
readily carry over to new situations.
\item A more fundamental problem is that the inequality \eqref{eq:momentmethod}
that forms the foundation for the moment method may fail to hold 
altogether for any $p\gg\log N$. To see why, suppose that
$\mathbf{P}[\|X^N\|\ge \|X_{\rm F}\|+\varepsilon]\ge N^{-c}$ for
some $\varepsilon,c>0$. Then 
$$
	\EE[\tr (X^N)^{2p}] \ge N^{-c} (\|X_{\rm F}\|+\varepsilon)^{2p},
$$
which precludes the validity of \eqref{eq:momentmethod} for $p\gg\log N$.
It was observed by Friedman \cite{Fri08} that this situation arises
already in the setting of random regular graphs due to the presence
with probability $N^{-c}$ of dense subgraphs called ``tangles''.
The application of the moment method in this setting requires
conditioning on the absence of tangles, which significantly complicates 
the analysis.
\end{itemize}

\subsection{A new approach}
\label{sec:newappr}

The approach of this paper is also based on moment computations, which are 
however used in an entirely different manner. As we will explain below, 
our method deduces norm convergence from moments by first letting 
$N\to\infty$ and then $p\to\infty$, bypassing the challenges of the moment 
method.

\subsubsection*{\bf Inputs}

Our approach requires two basic inputs that we presently describe.

Let $h\in\mathcal{P}_q$ be a polynomial of degree at most $q$. 
Then $\EE[\ntrN h(X^N)]$ is a linear combination of the moments 
$\EE[\ntrN (X^N)^p]$ for $p\le q$. 
We will exploit the fact that in many situations, the moments
are rational functions of $\frac{1}{N}$:
\begin{equation}
\label{eq:introrational}
	\EE[\ntrN h(X^N)] = \Phi_h(\tfrac{1}{N}) =
	\frac{f_h(\frac{1}{N})}{g_q(\frac{1}{N})},
\end{equation}
where $f_h$ and $g_q$ are polynomials that depend only on $h$ and $q$, 
respectively. This phenomenon is very common; e.g., it arises
from Weingarten calculus \cite{Col22,Col23} or from genus 
expansions \cite[Chapter 1]{MS17}. For the random permutation models
considered in this paper, the relevant properties are recalled
in section \ref{sec:words}.

In general, the function $\Phi_h$ is extremely complicated. However, our 
methods will use only soft information on its structure: we require upper 
bounds on the degrees of $f_h$ and $g_q$ (which are 
proportional to $q$ for the models considered in this paper), and we must 
show that $g_q$ does not vanish near zero. Both properties can be read off 
almost immediately from a combinatorial expression for $\Phi_h$.

The second input to our method is the asymptotic expansion 
as $N\to\infty$
\begin{equation}
\label{eq:introanalytic}
	\EE[\ntrN h(X^N)] = 
	\nu_0(h) + \frac{\nu_1(h)}{N} + O\bigg(\frac{1}{N^2}\bigg),
\end{equation}
where $\nu_0$ and $\nu_1$ are linear functionals on the space
$\mathcal{P}$ of real polynomials. Clearly
\begin{align*}
	\nu_0(h) &= \Phi_h(0) = \lim_{N\to\infty} \EE[\ntrN h(X^N)] =
	\tau(h(X_{\rm F})),\\
	\nu_1(h) &= \Phi_h'(0) = \lim_{N\to\infty} N\big(
	\EE[\ntrN h(X^N)]- \tau(h(X_{\rm F}))\big)
\end{align*}
are defined by the lowest-order asymptotics of the moments.
We will exploit that
simple combinatorial expressions for $\nu_0$ and $\nu_1$ are readily 
accessible in practice.

\begin{rem}
\label{rem:extensionisnontrivial}
It is evident from the above formulas that $\nu_0(h)=\int h(x)\,\nu_0(dx)$ 
coincides with the spectral distribution $\nu_0$ of $X_{\rm F}$, so we 
know \emph{a priori} that it is defined by a probability measure. In 
particular, even though $\nu_0(h)$ appears in \eqref{eq:introanalytic} 
only for polynomial test functions $h$, its definition automatically 
extends to any continuous function $h$. In contrast, there is no reason 
\emph{a priori} to expect that $\nu_1(h)$ can be meaningfully defined
for non-polynomial test functions $h$: there could be
functions $h$ for which weak convergence holds at a slower rate than 
$\frac{1}{N}$, in which case the expansion \eqref{eq:introanalytic} 
fails to hold. 
That \eqref{eq:introanalytic} and the definition of
$\nu_1(h)$ nonetheless extend to 
\emph{smooth} functions $h$ will be a key output of our method. 
\end{rem}

\subsubsection*{\bf Outline}

Our basic aim is to achieve the following.\vspace*{.05cm}
\begin{enumerate}[leftmargin=*,label=\!\!(\alph*)\!]
\addtolength{\itemsep}{.2cm}
\item We aim to show that the validity of \eqref{eq:introanalytic} 
extends from polynomials to smooth functions $h$. In particular,
we will show that $\nu_1$ extends to a compactly supported distribution
(in the sense of Schwartz, cf.\ Definition \ref{defn:distr}).
\item We aim to show that $\supp \nu_1 \subseteq [-\|X_{\rm F}\|,\|X_{\rm 
F}\|]$. \!(That $\supp \nu_0 \subseteq [-\|X_{\rm F}\|,\|X_{\rm
F}\|]$ holds automatically as $\nu_0$ is the spectral distribution of
$X_{\rm F}$.)
\end{enumerate}
\vspace*{.05cm}
A bound on the norm then
follows easily. Let $\chi:\mathbb{R}\to[0,1]$ be a smooth function 
so that $\chi=0$ on $[-\|X_{\rm F}\|-\frac{\varepsilon}{2},\|X_{\rm 
F}\|+\frac{\varepsilon}{2}]$ 
and $\chi=1$
on $[-\|X_{\rm F}\|-\varepsilon,\|X_{\rm F}\|+\varepsilon]^c$.
Then $\nu_0(\chi)=\nu_1(\chi)=0$ by (b), and thus (a) yields
$$
	\mathbf{P}[\|X^N\|\ge \|X_{\rm F}\|+\varepsilon] \le 
	\EE[\tr \chi(X^N)]=O\bigg(\frac{1}{N}\bigg).
$$
As $\varepsilon>0$ was arbitrary,
$\|X^N\|\le \|X_{\rm F}\|+o(1)$ in probability.

We now explain the steps that will be used to prove (a) and (b). 


\subsubsection*{\bf Step 1: the polynomial method}

At the heart of our approach lies a general method to obtain a
quantitative form of \eqref{eq:introanalytic}: we will show that
\begin{equation}
\label{eq:intropoly}
	\bigg|\EE[\ntrN h(X^N)]-\nu_0(h) - \frac{\nu_1(h)}{N}\bigg|
	\lesssim \frac{q^8}{N^2}\|h\|_{C^0[-K,K]}
\end{equation}
for any $h\in\mathcal{P}_q$,
where $\|X^N\|\le K$ a.s.\ for all $N$. While 
achieving such a bound by previous methods would require hard analysis, it 
arises here from a soft argument: we can ``upgrade'' an asymptotic 
expansion \eqref{eq:introanalytic} to a strong nonasymptotic bound 
\eqref{eq:intropoly} merely by virtue of the fact that $\Phi_h$ is 
rational.

To this end, observe that the left-hand side of 
\eqref{eq:intropoly} is nothing other than the remainder in the 
first-order Taylor expansion of $\Phi_h$ at zero, so that
$$
	\big|\Phi_h(\tfrac{1}{N}) - \Phi_h(0) - \tfrac{1}{N}\Phi_h'(0)
	\big| \le \tfrac{1}{2N^2}\|\Phi_h''\|_{C^0[0,\frac{1}{N}]}.
$$
This bound appears useless at first sight, as it relies on the behavior 
of $\Phi_h(x)$ for $x\not\in J:=\{\frac{1}{N}:N\in\mathbb{N}\}$ where its 
spectral interpretation is lost. However, as $\Phi_h$ is rational, we 
can control its behavior by its values on $J$ alone
by means of classical 
inequalities for arbitrary univariate polynomials $f$ of degree $q$:
\vspace*{.5mm}
\begin{itemize}[leftmargin=*]
\addtolength{\itemsep}{.2cm}
\item The inequality
$\|f'\|_{C^0[-1,1]}\le q^2\|f\|_{C^0[-1,1]}$ 
of A.\ and V.\ Markov (Lemma \ref{lem:markov}).
\item A corollary of the Markov inequality that
$\|f\|_{C^0[-1,1]}\lesssim\sup_{x\in I}|f(x)|$ for any set
$I\subset[-1,1]$ with $O(\frac{1}{q^2})$ spacing between its points
(Lemma \ref{lem:interpol}).
\end{itemize}
\vspace*{.5mm}
Applying these inequalities to the numerator and denominator of
$\Phi_h$ yields \eqref{eq:intropoly} with minimal effort.
We emphasize that the only
inputs used here are upper bounds on the degrees of $f_h$ and $g_q$ 
in \eqref{eq:introrational}, and that $g_q$ does not vanish near zero.

\subsubsection*{\bf Step 2: the extension problem}

We now aim to extend \eqref{eq:intropoly} to $h\in C^\infty$.
This is not merely a technical issue: while $\EE[\ntrN 
h(X^N)]$ and $\nu_0(h)=\tau(h(X_{\rm F}))$ are defined for $h\in 
C^\infty$ by functional calculus, it is unclear that $\nu_1(h)$ can 
even be meaningfully defined when $h$ is not a polynomial (cf.\ 
Remark \ref{rem:extensionisnontrivial}).

In order to address these issues, we must replace the degree-dependent 
bound \eqref{eq:intropoly} by a bound that depends only on the 
smoothness of the polynomial $h$. This is achieved as follows.
Rather than applying \eqref{eq:intropoly} 
directly to $h\in\mathcal{P}_q$, we expand
$$
	h(x)=\sum_{j=0}^q a_j T_j(K^{-1}x)
$$
in terms of Chebyshev polynomials of the first kind $T_j$, and apply
\eqref{eq:intropoly} to each $T_j$ individually. As Chebyshev polynomials
are uniformly bounded by $1$, this replaces $q^8\|h\|_{C^0[-K,K]}$ by 
$\sum_{j=1}^q j^8|a_j|\lesssim \|h\|_{C^9[-K,K]}$ in \eqref{eq:intropoly}, 
where the latter inequality follows by classical Fourier analysis. 
We therefore obtain
\begin{equation}
\label{eq:introsm}
	\bigg|\EE[\ntrN h(X^N)]-\nu_0(h) - \frac{\nu_1(h)}{N}\bigg|
	\lesssim \frac{1}{N^2}\|h\|_{C^9[-K,K]}
\end{equation}
for every polynomial $h\in\mathcal{P}$. (We will in fact use a stronger 
inequality of Zygmund, cf.\ Lemma \ref{lem:absconv}, that achieves better 
rates in our main results.)

Since the inequality \eqref{eq:introsm} no longer depends on the degree of 
the polynomial $h$, it extends to arbitrary smooth functions $h$ as 
polynomials are dense in $C^\infty$. In particular, it ensures that the 
left-hand side extends uniquely to a compactly supported distribution, so 
that $\nu_1(h)$ can be defined for any $h\in C^\infty$.

\subsubsection*{\bf Step 3: the asymptotic moment method}

It remains to bound the support of $\nu_1$. To this end, we make 
fundamental use of a key property of compactly supported distributions:
their support is bounded by the exponential growth rate of their moments
(Lemma \ref{lem:distsupp}). In particular, if we define
$$
	\rho =
	\limsup_{p\to\infty} |\nu_1(x^p)|^{1/p} =
	\limsup_{p\to\infty}\lim_{N\to\infty}
	\big| N\big(
        \EE[\ntrN (X^N)^p]- \tau(X_{\rm F}^p)\big)\big|^{1/p},
$$
then $\supp\nu_1 \subseteq [-\rho,\rho]$. Thus our method 
ultimately reduces to a form of the moment method, but where we first let 
$N\to\infty$ and only then $p\to\infty$.

In contrast to the moments of the random matrix $X^N$, the moments of 
$\nu_1$ turn out to be easy to analyze for the models 
considered in this paper using a simple idea that is inspired by 
earlier work of Friedman \cite[Lemma~2.4]{Fri03}: even though $\nu_1$ does 
not have a clear spectral interpretation, its moments can be expressed as 
a sum of products of matrix elements of powers of $X_{\rm F}$. As the sum 
only has polynomially many terms, the desired bound $\rho\le\|X_{\rm F}\|$ 
follows readily.

\subsection{The role of cancellations}

A surprising feature of our approach is that tangles, which form a 
fundamental obstruction to the moment method, are completely ignored. This 
seems unexpected, as the method is based on an estimate 
\eqref{eq:intropoly} for traces of high degree polynomials which are 
merely linear combinations of moments. Indeed, as was explained in section 
\ref{sec:prelude}, the presence of tangles causes the random matrix 
moments of large degree to be \emph{exponentially} larger than their 
limiting values, that is, $\mathbf{E}[\ntrN (X^N)^{2p}] \ge 
e^{cp}\tau(X_{\rm F}^{2p})$ for $p\gg\log N$. In contrast, 
\eqref{eq:intropoly} yields a bound on the difference between linear 
combinations of the random matrix moments and their limiting values that 
is only \emph{polynomial} in the degree.

The key feature of \eqref{eq:intropoly} is that it involves a uniform 
bound $\|h\|_{C^0[-K,K]}$ on the test function $h$. It therefore yields no 
useful information on moments of order $p\gg\log N$, as 
$\|h\|_{C^0[-K,K]}=K^p$ is exponential in $p$ for $h(x)=x^p$. However, it 
yields a powerful bound for polynomials $h$ that are uniformly bounded on 
the interval $[-K,K]$ independently of their degree. The estimate 
\eqref{eq:intropoly} therefore reveals a new phenomenon: the effect of 
tangles on the individual moments cancels out when they are combined to 
form bounded test functions $h$. One of the key features of the polynomial 
method is that it is able to capture these cancellations.

\subsection{Related work}

The observation that norm bounds follow from the validity of 
\eqref{eq:introanalytic} for \emph{smooth} functions $h$ and a bound on 
first-order support $\supp\nu_1$ has been widely used since the works of 
Haagerup and Thorbj{\o}rnsen \cite{HT05} and Schultz \cite{Sch05} (see 
also the work of Parraud \cite{Par23} where higher-order expansions are 
considered). However, these works rely heavily on analytic and 
operator-algebraic tools that are not available for the kind of models 
considered in this paper. What is fundamentally new here is that our 
method achieves this aim using only simple moment computations, which are 
much more broadly applicable.

The polynomial method that lies at the heart of our approach is inspired 
by work in complexity theory \cite{Aar08}. We are aware of little 
precedent for the use of this method in a random matrix context, beside a 
distantly related idea that appears in Bourgain and Tzafriri \cite[Theorem 
2.3]{BT91}. The first author used a variant of this idea in concurrent 
work to establish weak convergence of certain random unitary matrices that 
arise in quantum information theory \cite{Che24a,Che24b}.

We emphasize that the appearance of Chebyshev polynomials in Step 2 above 
is unrelated to their connection with nonbacktracking walks that is used, 
for example, in \cite[Lemma 2.3]{Fri08}. Indeed, our Chebyshev polynomials 
are normalized by the a.s.\ bound $K$ on $\|X^N\|$ (e.g., $2d$ in 
Friedman's theorem) as opposed to the support of the limiting spectrum 
$\|X_{\rm F}\|$ ($2\sqrt{2d-1}$ in Friedman's theorem).

Finally, we mention the interesting works \cite{BCG22,CLZ24} that develop 
a method to bound the spectral radius of certain non-Hermitian matrices
using $N\to\infty$ moment asymptotics with fixed $p$, similarly avoiding 
the challenges of the moment method in that setting. However, the methods
developed in these works are unrelated to our approach and are not 
applicable to the problems considered in this paper.

\section{Main results}
\label{sec:main}

In this section, we formulate and discuss the main results of this paper. 
As our primary aim is to achieve strong convergence, we will focus the 
presentation on upper bounds on the norm as explained in section 
\ref{sec:intro}. Let us note, however, that a byproduct of the analysis 
will also yield a quantitative form of weak convergence, which is of 
independent interest; see Corollary \ref{cor:quantwaf} below.

\subsection{Preliminaries}
\label{sec:prelim}

Before we turn to our main results, we must briefly recall some basic 
facts about random permutation matrices and their limiting model.
The following definitions will remain in force throughout this paper.

\begin{defn}
Let $\bar S_1^N,\ldots \bar S_d^N$ be 
i.i.d.\ random permutation matrices of dimension $N$, and denote by 
$S_i^N:= \bar S_i^N|_{\{1_N\}^\perp}$ their restriction to the invariant 
subspace $\{1_N\}^\perp\subset\mathbb{C}^N$. We will often write 
$\boldsymbol{S}^N=(S_1^N,\ldots,S_d^N)$ and
$\boldsymbol{S}^{N*}=(S_1^{N*},\ldots,S_d^{N*})$.
\end{defn}

\begin{defn}
Let $\boldsymbol{s}=(s_1,\ldots,s_d)$ be defined by $s_i:=\lambda(g_i)$, 
where $g_1,\ldots,g_d$ are the free generators of $\mathbf{F}_d$ and 
$\lambda:\mathbf{F}_d \to B(l^2(\mathbf{F}_d))$ is the left-regular 
representation defined by 
$\lambda(g)\delta_h=\delta_{gh}$. Define the vector state $\tau(a) := 
\langle \delta_e,a\delta_e\rangle$ on $B(l^2(\mathbf{F}_d))$.
\end{defn}

The basic weak convergence property of random permutation matrices, due to 
Nica \cite{Nic93} (see Corollary \ref{cor:waf} for a short proof), states 
that
$$
	\lim_{N\to\infty} 
	\mathbf{E}\big[\ntrN P(\boldsymbol{S}^N,\boldsymbol{S}^{N*})\big]
	=
	\tau\big(P(\boldsymbol{s},\boldsymbol{s}^*)\big)
$$
for every noncommutative polynomial $P$. This property plays the role of
\eqref{eq:weakcv} in section \ref{sec:outline}. The aim of the strong 
convergence problem is to prove that this convergence holds not only for 
the trace but also for the norm.

The basic inputs to the methods of this paper, as described in section 
\ref{sec:newappr}, are well known in the present setting.
They will be reviewed in section \ref{sec:words} below.

\begin{rem}
Even though $S_i^N$ are $(N-1)$-dimensional matrices defined on 
$\{1_N\}^\perp$, we will normalize the trace $\ntrN$ by $N$ rather than by 
$N-1$ as this leads to cleaner expressions for the rational functions that 
arise in the proof. This makes no difference to our results, 
and is mainly done for notational convenience.
\end{rem}

\subsection{Random regular graphs}
\label{sec:rregmain}

Let $\bar A^N$ be the adjacency matrix of the random $2d$-regular graph
with $N$ vertices defined in section \ref{sec:intro}. Then 
$A^N := \bar A^N|_{\{1_N\}^\perp}$ is defined by the linear polynomial
of random permutation matrices
$$
	A^N := S_1^N + S_1^{N*} + \cdots + S_d^N + S_d^{N*},
$$
and the associated limiting model is
$$
	A_{\rm F} := s_1 + s_1^* + \cdots + s_d + s_d^*.
$$
Note that $A_{\rm F}$ is nothing other than the adjacency matrix of the 
Cayley graph of $\mathbf{F}_d$ generated by $g_1,\ldots,g_d$ and their 
inverses, that is, of the $2d$-regular tree. 

It is a classical fact due to 
Kesten \cite{Kes59} that $\|A_{\rm F}\|=2\sqrt{2d-1}$. That
$$
	\|A^N\| \le 2\sqrt{2d-1} + o(1)\quad\text{with probability }
	1-o(1)
$$
is due to Friedman \cite{Fri08}. Friedman's proof was simplified by 
Bordenave \cite{Bor20}, and a new proof with an improved convergence
rate was given by Huang and Yau \cite{HY24}. Very recently, the latter 
approach was further developed in the impressive works of Huang, McKenzie 
and Yau \cite{HMY24,HMY25} to achieve the optimal convergence rate.

\subsubsection{An effective Friedman theorem}

As a first illustration of the methods of this paper, we will give a 
short new proof of Friedman's theorem in section \ref{sec:friedman}.
A direct implementation of the approach described in
section \ref{sec:outline} yields the following.

\begin{thm}
\label{thm:friedman} 
For every $d\ge 2$, $N\ge 1$, and 
$\varepsilon < 2d-2\sqrt{2d-1}$,  we have
$$
	\mathbf{P}\big[\|A^N\|\ge 2\sqrt{2d-1}+\varepsilon\big]
	\lesssim
	\frac{1}{N}\, 
	\bigg(\frac{d\log d}{\varepsilon}\bigg)^{8}
	\log\bigg(\frac{2ed}{\varepsilon}\bigg).
$$
\end{thm}

Theorem \ref{thm:friedman} implies that
when $d$ is fixed as $N\to\infty$, we have\footnote{%
The notation $Z_N=\Opr(z_N)$ denotes that $\{Z_N/z_N\}_{N\ge 1}$ is 
bounded in probability.}
$$
	\|A^N\| \le 2\sqrt{2d-1} + 
	\Opr\big(\big(\tfrac{\log N}{N}\big)^{1/8}\big).
$$
This rate falls short of the optimal $N^{-2/3}$ rate in Friedman's theorem 
that was very recently established in \cite{HMY24,HMY25}. However, the 
methods of \cite{HY24,HMY24,HMY25} rely heavily on the special structure 
of random regular graphs, and it is unclear at present whether they could 
be applied to the study of strong convergence. In contrast, the methods of 
this paper will achieve the same rate for arbitrary polynomials of random 
permutation matrices, see section \ref{sec:strongmain} below.

The quantitative nature of Theorem \ref{thm:friedman} also enables us to 
consider what happens when $d,N\to\infty$ simultaneously. It is an old 
question of Vu \cite[\S 5]{Vu08} whether $\|A^N\|= (1+o(1))2\sqrt{2d-1}$ 
with probability $1-o(1)$ remains valid when $d,N\to\infty$ in an 
arbitrary manner. We can now settle this question for the permutation 
model of random regular graphs that is considered here (see Remark 
\ref{rem:contig} below for a brief discussion of other models of random 
regular graphs).

\begin{cor}
\label{cor:vu}
$\|A^N\|=(1+o(1))2\sqrt{2d-1}$ with probability $1-o(1)$ whenever
$N\to\infty$ and $d=d(N)$ depends on $N$ in an arbitrary manner.
\end{cor}

\begin{proof}
That the conclusion holds for $d\ge (\log N)^5$ was proved in
\cite[\S 3.2.2]{BV23}. In the complementary regime $d\le (\log N)^5$,
Theorem \ref{thm:friedman} readily yields the upper bound
$\|A^N\|\le(1+o(1))2\sqrt{2d-1}$ with probability $1-o(1)$, 
while the corresponding lower bound
follows from the Alon--Boppana theorem \cite{Nil91}.
\end{proof}

\subsubsection{The staircase theorem}

As was explained in section \ref{sec:outline}, the approach of this paper 
only requires an understanding of the first-order term $\nu_1$ in the 
$\frac{1}{N}$-expansion of the moments. However, in the setting of random 
regular graphs, a detailed understanding of the higher-order terms was 
achieved by Puder \cite{Pud15} using methods of combinatorial group 
theory. When such additional information is available, the approach of 
this paper is readily adapted to achieve stronger results.

The following theorem will be proved in section \ref{sec:puder} by taking 
full advantage of the results of \cite{Pud15}. We emphasize that this is 
the only part of this paper where we will use asymptotics of expected 
traces beyond the lowest order.

\begin{thm}
\label{thm:puder}
Define 
$$
	\rho_m := \begin{cases}
	2\sqrt{2d-1} & \text{for } 2m-1 \le \sqrt{2d-1},\\
	2m-1 + \tfrac{2d-1}{2m-1}
	& \text{for }2m-1>\sqrt{2d-1},
	\end{cases}
$$
and let
$m_* := \lfloor \frac{1}{2}(\sqrt{2d-1} + 1) \rfloor$
be the largest integer $m$ so that $2m-1\le\sqrt{2d-1}$.

Then for every $d\ge 2$, $m_*\le m\le d-1$, and
$0<\varepsilon<\rho_{m+1}-\rho_m$, we have
\begin{align*}
&	
	\mathbf{P}\big[\|A^N\| \ge
	\rho_m+\varepsilon\big]
	\le
	\frac{C_d}{N^m}
	\frac{1}{\varepsilon^{4(m+1)}}
	\log
	\bigg(\frac{2e}{\varepsilon}\bigg)
\\
\intertext{for all $N\ge 1$, and}
&
	\mathbf{P}\big[\|A^N\| \ge
        \rho_m+\varepsilon\big] \ge
	\frac{1-o(1)}{N^m}
\end{align*}
as $N\to\infty$. Here $C_d$ is a constant that depends on $d$ only.
\end{thm}

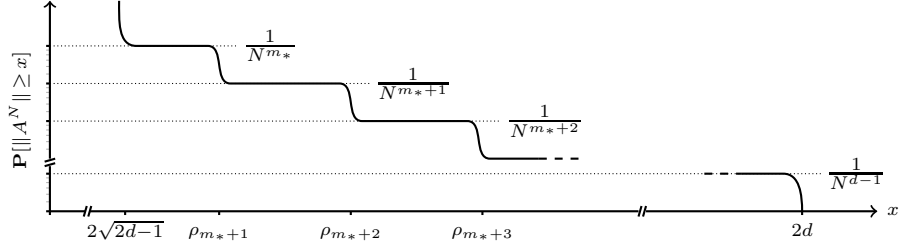
\begin{figure} \centering \begin{tikzpicture}

\begin{scope}[shift={(0,-0.5)}]
\begin{scope}[shift={(0,-0.3)}]

\foreach \i in {4,5,6}
{
	\draw[thick] (0,\i/2) to (-.05,\i/2);

	\foreach \j in {1,...,5}
	{
	\draw[color=black!30] (0,{\i/2+ln(1-0.632*\j/6)/2}) -- (-0.05,{\i/2+ln(1-0.632*\j/6)/2});
	}
}

\draw[color=black!100,densely dotted] (0,3) -- (2.5,3);
\draw (2.45,3) node[right] {$\tfrac{1}{N^{m_*}}$};

\draw[color=black!100,densely dotted] (0,2.5) -- (4.25,2.5);
\draw (4.2,2.5) node[right] {$\tfrac{1}{N^{m_*+1}}$};

\draw[color=black!100,densely dotted] (0,2) -- (6,2);
\draw (5.95,2) node[right] {$\tfrac{1}{N^{m_*+2}}$};

\draw[thick] (0.92,3.6) to[out=270,in=180] (1.15,3) 
to[out=0,in=180] (2.1,3) to[out=0,in=180] (2.4,2.5) 
to[out=0,in=180] (3.85,2.5) to[out=0,in=180] (4.15,2)
to[out=0,in=180] (5.55,2) to[out=0,in=180] (5.85,1.5)
to (6.5,1.5);

\draw[thick,dashed] (6.5,1.5) -- (7.1,1.5);

\end{scope}

	\draw[thick] (0,1) to (-.05,1);

	\foreach \j in {1,...,5}
	{
	\draw[color=black!30] (0,{1+ln(1-0.632*\j/6)/2}) -- (-0.05,{1+ln(1-0.632*\j/6)/2});
	}

\draw[thick,->] (0,1.2) to (0,3.3);
\draw[thick] (0,0.5) to (0,1.1);
\draw[thick] (-0.07,1.1) to (0.07,1.14);
\draw[thick] (-0.07,1.16) to (0.07,1.2);

\end{scope}

\begin{scope}[shift={(-1,0)}]

\draw[color=black!100,densely dotted] (1,0.5) -- (11.25,0.5);
\draw (11.2,0.5) node[right] {$\tfrac{1}{N^{d-1}}$};

\draw[thick,dashed] (9.7,0.5) -- (10.2,0.5);

\draw[thick] (10.2,0.5) 
to[out=0,in=180] (10.75,0.5) to[out=0,in=90] (11,0);

\draw[thick] (11,0) -- (11,-.05);
\draw (11,-.25) node {$\scriptstyle 2d$};

\end{scope}

\draw[thick] (0,0) to (0,-0.05);
\draw[thick] (0,0) to (-0.05,0);

\draw[thick] (0,0) to (0.47,0);
\draw[thick] (0.49,0.07) to (0.45,-0.07);
\draw[thick] (0.55,0.07) to (0.51,-0.07);

\draw[thick] (0.53,0) to (7.825,0);
\draw[thick,->] (7.925,0) to (11,0) node[right] {$\scriptstyle x$};

\draw[thick] (1,0) -- (1,-.05);
\draw (1,-.25) node {$\scriptstyle 2\sqrt{2d-1}$};

\draw[thick] (2.25,0) -- (2.25,-.05);
\draw (2.25,-.3) node {$\scriptstyle \rho_{m_*+1}$};

\draw[thick] (4,0) -- (4,-.05);
\draw (4,-.3) node {$\scriptstyle \rho_{m_*+2}$};

\draw[thick] (5.75,0) -- (5.75,-.05);
\draw (5.75,-.3) node {$\scriptstyle \rho_{m_*+3}$};

\draw[thick] (7.825,-0.07) to (7.865,0.07);
\draw[thick] (7.885,-0.07) to (7.925,0.07);

\draw (-.4,1.35) node[rotate=90]
{$\scriptstyle \mathbf{P}[\|A^N\|\,\ge\, x]$};

\end{tikzpicture}
\caption{Staircase pattern of the tail probabilities of 
$\|A^N\|$.\label{fig:staircase}}
\end{figure}
Theorem \ref{thm:puder} reveals an unusual staircase pattern of the large 
deviations of $\|A^N\|$, which is illustrated in Figure 
\ref{fig:staircase}. The lower bound, which follows from an elementary 
argument of Friedman \cite[Theorem 2.11]{Fri08}, arises due to the 
presence of tangles: Friedman shows that the presence of a vertex with 
$m>m_*$ self-loops, which happens with probability $\approx N^{1-m}$, 
gives rise to an outlier in the spectrum at location $\approx \rho_m$. The 
fact that our upper bound precisely matches this behavior shows that the 
tail probabilities of $\|A^N\|$ are completely dominated by these events. 

\begin{rem}
Let us highlight two further consequences of 
Theorem \ref{thm:puder}.
\vspace*{.5mm}
\begin{enumerate}[leftmargin=*,label=\arabic*.]
\addtolength{\itemsep}{.1cm}
\item Theorem \ref{thm:puder} shows that
$\|A^N\|\le 2\sqrt{2d-1}+\varepsilon$ with probability at least
$1-\frac{c}{N^{m_*}}$ for any $\varepsilon>0$, closing the gap
between the upper and lower bounds in the main result of
Friedman's original monograph \cite[p.\ 2]{Fri08}. This was
previously achieved by Friedman and Kohler \cite{FK14} 
by a refinment of Friedman's methods.
\item 
The lower bound of Theorem 
\ref{thm:puder} shows that the $O(\frac{1}{N})$ probability bound of
Theorem \ref{thm:friedman} cannot be improved in general for fixed 
$\varepsilon>0$, as this bound is sharp when $m_*=1$ (that is, when 
$2\le d\le 4$).
\end{enumerate}
\end{rem}

\begin{rem}
\label{rem:contig}
The random regular graph model considered in this paper is known as the 
\emph{permutation model}. Several other models of random regular graphs 
are also considered in the literature. All these models are known to be 
contiguous to the permutation model as $N\to\infty$ for fixed degree $2d$, 
so that any statement that holds with probability $1-o(1)$ for the 
permutation model remains valid with probability $1-o(1)$ for the other 
models; cf.\ \cite[pp.\ 2--3]{Fri08} and the references therein.

It should be emphasized, however, that low probability events are not 
preserved by contiguity. In particular, the detailed quantitative picture 
in Theorem \ref{thm:puder} is specific to the permutation model. The 
corresponding picture for other models of random regular graphs must be 
investigated on a case by case basis.

Similarly, the different models of random regular graphs are no longer 
contiguous when the degree diverges $d\to\infty$, so that high degree 
results as in Corollary \ref{cor:vu} must be investigated separately for 
each model. Partial results on this problem for the uniform model 
of random regular graphs may be found in \cite{BHKY20,Sar23,He24}.
\end{rem}

\subsection{Strong convergence of random permutation matrices}
\label{sec:strongmain}

The adjacency matrix of a random regular graph is one very special example
of a polynomial of random permutation matrices. The much more general 
fact that the norm of \emph{every} noncommutative polynomial of random 
permutation matrices converges to that of its limiting model is an 
important result due to Bordenave and Collins \cite{BC19,BC24}. Here
we give a short proof that yields an effective form of this result.

We will formulate our strong convergence results for general 
noncommutative polynomials $P\in 
\mathrm{M}_D(\mathbb{C})\otimes \mathbb{C}\langle\boldsymbol{s},
\boldsymbol{s}^*\rangle$ with matrix coefficients, that is,
$$
	P(\boldsymbol{s},\boldsymbol{s}^*) =
	\sum_w A_w \otimes w(\boldsymbol{s},\boldsymbol{s}^*),
$$
where the sum is over a finite set of words $w$ in the symbols
$s_1,\ldots,s_d,s_1^*,\ldots,s_d^*$ and $A_w\in \mathrm{M}_D(\mathbb{C})$ 
are matrix coefficients. The case of scalar coefficients is recovered for
$D=1$, but the more general setting considered here arises often in 
applications (e.g., in the study of random lifts \cite{BC19}).
For any such polynomial, we define 
the norm
$$
	\|P\|_{\mathrm{M}_D(\mathbb{C})\otimes C^*(\mathbf{F}_d)} = 
	\sup_{\boldsymbol{U}} \|P(\boldsymbol{U},\boldsymbol{U}^*)\|,
$$
where the supremum is taken over all $d$-tuples of 
unitary matrices of any dimension. The notation used here agrees with the 
standard definition in terms of the norm in the full $C^*$-algebra of 
the free group, see, e.g., \cite[Theorem 7]{Cho80}.
However, this is not important for our purposes, and the reader may simply 
view the above as the definition of the norm. We emphasize that
$\|P\|_{\mathrm{M}_D(\mathbb{C})\otimes C^*(\mathbf{F}_d)}$ is not the 
same as $\|P(\boldsymbol{s},\boldsymbol{s}^*)\|$, which corresponds to the
reduced $C^*$-algebra of
the free group. For example, for $P(\boldsymbol{s},\boldsymbol{s}^*)=
s_1+s_1^*+\cdots+s_d+s_d^*$, we have 
$\|P(\boldsymbol{s},\boldsymbol{s}^*)\|=2\sqrt{2d-1}$ and
$\|P\|_{C^*(\mathbf{F}_d)} =2d$. In practice, 
$\|P\|_{\mathrm{M}_D(\mathbb{C})\otimes C^*(\mathbf{F}_d)}$ 
may be bounded
trivially by the sum of the norms of the matrix coefficients of $P$.

We now state our main result on strong convergence of random 
permutations.

\begin{thm}
\label{thm:strongperm}
Let $d\ge 2$, and let $P\in \mathrm{M}_D(\mathbb{C})
\otimes \mathbb{C}\langle\boldsymbol{s},\boldsymbol{s}^*\rangle$ be any 
self-adjoint noncommutative polynomial of degree $q_0$. Then we have
$$
	\mathbf{P}\big[\|P(\boldsymbol{S}^N,\boldsymbol{S}^{N*})\|\ge
	\|P(\boldsymbol{s},\boldsymbol{s}^*)\|+\varepsilon\big]
	\lesssim
	\frac{D}{N}
	\bigg(\frac{Kq_0\log d}{\varepsilon}\bigg)^{8}
	\log
	\bigg(\frac{eK}{\varepsilon}\bigg)
$$
for all $\varepsilon < K-\|P(\boldsymbol{s},\boldsymbol{s}^*)\|$,
where $K=\|P\|_{\mathrm{M}_D(\mathbb{C})\otimes C^*(\mathbf{F}_d)}$.
\end{thm}

This theorem will be proved in section \ref{sec:bc}. Note that the 
limiting norm $\|P(\boldsymbol{s},\boldsymbol{s}^*)\|$ can in principle be 
computed explicitly using the results of \cite{Leh99}.

\begin{rem}
The assumption that $P$ is self-adjoint entails no loss of generality: 
the analogous bound for a non-self-adjoint polynomial $P$ of degree 
$q_0$ follows by applying this result to the self-adjoint polynomial 
$P^*P$ of degree $2q_0$ (cf.\ Remark~\ref{rem:sanoloss}).
\end{rem}

Theorem \ref{thm:strongperm} shows that 
$\|P(\boldsymbol{S}^N,\boldsymbol{S}^{N*})\|\le 
\|P(\boldsymbol{s},\boldsymbol{s}^*)\|+ \Opr\big(\big(\frac{\log 
N}{N}\big)^{1/8}\big)$ for any polynomial $P$. This significantly improves 
the best known bound to date, due to Bordenave and Collins \cite[Theorem 
1.4]{BC24}, which yields fluctuations of order $\frac{\log\log N}{\log N}$.
Our bound can be directly substituted in applications, such as
to random lifts of graphs \cite[\S 1.5]{BC19}, to improve the best known 
convergence rates.

Let us note that for fixed $\varepsilon>0$, the tail probability of order 
$\frac{1}{N}$ in Theorem~\ref{thm:strongperm} cannot be improved in 
general, as is illustrated by Theorem \ref{thm:puder} with $d=2$.

\begin{rem}
While Theorem \ref{thm:strongperm} yields much stronger quantitative 
bounds for fixed $D$ then prior results, it can only be applied when 
$D=o(N)$. In contrast, it was recently shown in \cite[Corollary 1.5]{BC24} 
that strong convergence remains valid in the present setting even for $D$ 
as large as $N^{(\log N)^{1/2}}$.
Such bounds could be achieved using our methods if
the supports of 
the higher-order terms in the $\frac{1}{N}$-expansion of the moments are 
still bounded by $\|P(\boldsymbol{s},\boldsymbol{s}^*)\|$. While this is 
the case for continuous models such as GUE \cite{Par23}, it is 
simply false in the present setting: the proof of Theorem \ref{thm:puder} 
shows that tangles can already arise in the second-order term. It is an 
interesting question whether the approach of this paper can be combined 
with conditioning on the absence of tangles to achieve improved 
bounds.
\end{rem}

\subsection{Stable representations of the symmetric group}

Let $\pi_N:\mathbf{S}_N\to B(V_N)$ be a finite-dimensional unitary 
representation of the symmetric group $\mathbf{S}_N$. Then we can define 
random matrices $\Pi^N_1,\ldots,\Pi^N_d$ of dimension $\dim V_N$ as
\begin{equation}
\label{eq:stablermtx}
	\Pi_i^N := \pi_N(\sigma_i),
\end{equation}
where $\sigma_1,\ldots,\sigma_d$ are i.i.d.\ uniformly distributed 
elements of $\mathbf{S}_N$. When $\pi_N$ is the standard representation,
we recover the random permutation matrices $\Pi_i^N=S_i^N$ as a special 
case. Other representations, however, capture a much 
larger family of random matrix models. We will prove strong 
convergence for random matrices defined by any 
\emph{stable} representation of $\mathbf{S}_N$ (see section 
\ref{sec:stabrep}), which yields a far-reaching generalization of Theorem 
\ref{thm:strongperm}. This result is of interest for several 
reasons:\vspace*{.5mm}
\begin{itemize}[leftmargin=*]
\addtolength{\itemsep}{.2cm}
\item It provides many new examples of the strong convergence
phenomenon.
\item It shows that strong convergence can be achieved with much less 
randomness than in the standard representation: $\Pi^N_i$ uses the 
same number of random bits as $S_i^N$, but has much larger dimension 
($\dim V_N \asymp N^\alpha$ with $\alpha$ arbitrarily large). 
See \cite{BC20} for analogous questions in the context of the unitary 
group.
\item 
It provides new evidence in support of
long-standing questions on the expansion of random Cayley
graphs of the symmetric group, for which extensive numerical evidence
and conjectures are available; see \cite{RS19} and the references therein.
\item Random matrices defined by representations other than the standard
representation arise in various applications \cite{FJRST96,HH09}.
\end{itemize}

\begin{rem}
The question whether strong asymptotic freeness can be 
derandomized has been investigated in the theoretical computer science 
literature \cite{MOP19,OW20} by means of pseudorandom permutation 
matrices. The results of the present section suggest that high-dimensional 
representations of $\mathbf{S}_N$ can provide a different perspective 
on such questions. We omit a detailed discussion of the number of random 
bits needed by our bounds, as significantly stronger results in this 
direction were subsequently obtained by Cassidy \cite{Cas24} by combining 
the methods of this paper with new group-theoretic ideas; see Remark
\ref{rem:cassidy} below.
\end{rem}

\subsubsection{Stable representations}
\label{sec:stabrep}

The approach of this paper requires that the moments of the 
random matrices of interest are rational functions of $\frac{1}{N}$. For 
this to be the case, we cannot choose an arbitrary representation $\pi_N$ 
of $\mathbf{S}_N$ for each $N$. Instead, we will work with stable 
representations \cite{Far14,CEF15} that are defined consistently for 
different $N$. We briefly introduce the relevant notions here.

Following \cite{HP23}, denote by $\xi_i(\sigma)$ the number of fixed 
points of $\sigma^i$ for $\sigma\in\mathbf{S}_N$. The sequence 
$\xi_1(\sigma),\xi_2(\sigma),\ldots$ determines the conjugacy class of 
$\sigma$. If $\pi_N:\mathbf{S}_N\to B(V_N)$ is any finite-dimensional 
unitary representation, its character $\sigma\mapsto\tr \pi_N(\sigma)$ is 
a class function and can therefore be expressed as a polynomial of 
$\xi_1(\sigma),\xi_2(\sigma),\ldots$

\begin{defn}
\label{def:stable}
A finite-dimensional unitary representation
$\pi_N:\mathbf{S}_N\to B(V_N)$ of $\mathbf{S}_N$, defined for each $N\ge 
N_0$, is \emph{stable} if there exists $r\in\mathbb{N}$
and a polynomial $\varphi\in\mathbb{C}[x_1,\ldots,x_r]$ so that
$\tr \pi_N(\sigma) = \varphi(\xi_1(\sigma),\ldots,\xi_r(\sigma))$
for all $N\ge N_0,\sigma\in\mathbf{S}_N$.
\end{defn}

Thus stable representations are representations of $\mathbf{S}_N$ whose 
characters are defined by the same polynomial $\varphi$ for all $N\ge 
N_0$. For example, the standard representation is stable as it satisfies 
$\tr \pi_N=\xi_1-1$ for all $N$.

The irreducible stable representations of $\mathbf{S}_N$ can be 
constructed explicitly as follows. 
Fix a base partition $\lambda=(\lambda_1\ge\cdots\ge\lambda_\ell>0) 
\vdash|\lambda|$ of $|\lambda|=\sum_{i=1}^\ell\lambda_i$. Then for every 
$N\ge|\lambda|+\lambda_1$, the irreducible representation of 
$\mathbf{S}_N$ defined by
\begin{equation}
\label{eq:stablelambda}
	\lambda[N]=(N-|\lambda|\ge\lambda_1\ge\cdots\ge\lambda_\ell) \vdash N
\end{equation}
is stable. Moreover, it follows from \cite[Proposition B.2]{HP23} that 
every stable representation in the sense of Definition \ref{def:stable} is 
a direct sum of such irreducible representations defined by fixed 
base partitions $\lambda^{(1)},\ldots,\lambda^{(s)}$.

\subsubsection{Strong convergence}
\label{sec:stabmain}

Fix a stable representation $\pi_N:\mathbf{S}_N\to B(V_N)$ defined for 
$N\ge N_0$ by a character polynomial 
$\varphi\in\mathbb{C}[x_1,\ldots,x_r]$. We aim to prove strong convergence 
of the random matrices $\Pi^N_1,\ldots,\Pi^N_d$ defined by 
\eqref{eq:stablermtx}.

We will not require that $\pi_N$ is irreducible, but we assume it does
not contain the trivial representation.
The dimension of $\Pi_i^N$ is given by
$$
	D_N := \dim(V_N) = \tr \pi_N(e) = \varphi(N,N,\ldots,N).
$$
Thus $D_N$ is a polynomial in $N$; we denote its degree by $\alpha$,
so that $D_N\asymp N^\alpha$.

We now formulate our main result on strong convergence of 
$\boldsymbol{\Pi}^N=(\Pi^N_1,\ldots,\Pi^N_d)$, whose proof is contained
in section \ref{sec:stable} below. 

\begin{thm}
\label{thm:strongstable}
Let $d\ge 2$, and let $P\in \mathrm{M}_D(\mathbb{C})
\otimes \mathbb{C}\langle\boldsymbol{s},\boldsymbol{s}^*\rangle$ be any 
self-adjoint noncommutative polynomial of degree $q_0$. Then we have
$$
	\mathbf{P}\big[\|P(\boldsymbol{\Pi}^N,\boldsymbol{\Pi}^{N*})\|\ge
	\|P(\boldsymbol{s},\boldsymbol{s}^*)\|+\varepsilon\big]
	\le
	\frac{CD}{N}
	\bigg(\frac{K q_0\log d}{\varepsilon}\bigg)^{4(\alpha+1)}
	\log
	\bigg(\frac{eK}{\varepsilon}\bigg)
$$
for all $\varepsilon < K-\|P(\boldsymbol{s},\boldsymbol{s}^*)\|$ and
$N\ge N_0$. Here we define $K=\|P\|_{\mathrm{M}_D(\mathbb{C})\otimes 
C^*(\mathbf{F}_d)}$, and $C$ is a constant that depends on the choice of
stable representation.
\end{thm}

\begin{rem}
It is certainly possible to obtain an explicit expression for $C$ from 
the proof; we have suppressed the dependence on the choice of
stable representation for simplicity of presentation, and we did
not optimize this dependence in the proof.
\end{rem}

\begin{rem}
\label{rem:cassidy}
The bound $\|P(\boldsymbol{\Pi}^N,\boldsymbol{\Pi}^{N*})\| \le 
\|P(\boldsymbol{s},\boldsymbol{s}^*)\| + \Opr\big(\big(\frac{\log 
N}{N}\big)^{1/4(\alpha+1)}\big)$ that follows from Theorem 
\ref{thm:strongstable} becomes weaker when we consider stable 
representations of increasingly large dimension $D_N\asymp N^\alpha$. This 
is not a restriction of our method, however, but rather reflects a gap in 
the understanding of stable representations at the time that this paper
was written \cite[Conjecture 1.8]{HP23}. Important progress in this 
direction, recently obtained by Cassidy \cite{Cas24}, yields an
improved probability bound where $\frac{1}{N\varepsilon^{4(\alpha+1)}}$ 
is replaced by $\frac{1}{N^\alpha \varepsilon^{8\alpha}}$, resulting in a 
convergence rate that is independent of $\alpha$. As is discussed in
\cite{Cas24}, this yields strong convergence uniformly 
for all stable representations with $\alpha\le N^{\frac{1}{12}-\delta}$
for any $\delta>0$.
\end{rem}

Theorem \ref{thm:strongstable} can be readily applied to concrete 
situations. For example, it implies that random 
$2d$-regular Schreier graphs defined by the action of $\mathbf{S}_N$ on 
$k$-tuples of distinct elements of $\{1,\ldots,N\}$ have second 
eigenvalue $2\sqrt{2d-1}+o(1)$ with probability $1-o(1)$, settling a 
question discussed in \cite[\S 1.4]{HP23}. Indeed, it is not difficult 
to see (cf.\ \cite[\S 8]{HP23}) that the restricted adjacency matrix
$A^N=\bar A^N|_{\{1^\perp\}}$ of such a random graph can be represented as
$$
	A^N = \Pi_1^N + \Pi_1^{N*} + \cdots \Pi_d^N + \Pi_d^{N*}
$$
for some stable representation of $\mathbf{S}_N$ (which depends on $k$)
that does not contain the trivial representation,
so that the conclusion follows immediately as a special case
of Theorem \ref{thm:strongstable}. Applications of this model may be found 
in \cite{FJRST96,HH09}. (Let us note for completeness that the case $k=1$ 
reduces to Friedman's theorem, while the case $k=2$ was previously studied
by Bordenave and Collins \cite[\S 1.6]{BC19}.)

\section{Basic tools}
\label{sec:basic}

The aim of this section is to introduce the general facts on polynomials 
and compactly supported distributions that form the basis for the methods 
of this paper. While most of the tools in this section follow readily from 
known results, it is useful to state them in the particular form in which 
they will be needed.

\subsection{Markov inequalities}

One of the key tools that will be used in our analysis is the following
classical inequality of A.\ Markov and V.\ Markov. Here we recall that
$(2k-1)!!:=(2k-1)(2k-3)\cdots 5\cdot 3\cdot 1 =\frac{(2k)!}{2^k k!}$.

\begin{lem}[Markov inequality]
\label{lem:markov}
Let $h\in\mathcal{P}_q$ and $a>0$, $m\in\mathbb{N}$. Then we have
$$
        \|h^{(m)}\|_{C^0[0,a]}
        \le \frac{1}{(2m-1)!!} 
        \bigg(\frac{2q^2}{a}\bigg)^m
	\|h\|_{C^0[0,a]}.
$$
\end{lem}

\begin{proof}
Apply \cite[p.\ 256(d)]{BE95} to
$P(x)=h(\frac{a}{2}x + \frac{a}{2})$.
\end{proof}

Two basic issues will arise in our applications of this inequality. First, 
we will not be able to control the relevant functions on the entire 
interval $[0,a]$, but only on a discrete subset thereof. This issue 
will be addressed using a standard interpolation inequality that is itself 
a direct consequence of the Markov inequality.

\begin{lem}[Interpolation]
\label{lem:interpol}
Let $h\in\mathcal{P}_q$, and fix a subset $I\subseteq[0,a]$ such that 
$[0,a]\subseteq I + [-\frac{a}{4q^2},\frac{a}{4q^2}]$. Then we have 
$$
        \|h\|_{C^0[0,a]} \le
        2\sup_{x\in I}|h(x)|.
$$
\end{lem}

\begin{proof}
Apply \cite[Lemma 3(i), p.\ 91]{Che98} to
$P(x)=h(\frac{a}{2}x + \frac{a}{2})$.
\end{proof}

The second issue is that we will aim to apply these inequalities to 
rational functions rather than to polynomials. However, in the cases of 
interest to us, the denominator of the rational function will be nearly 
constant. The following lemma extends the Markov inequality
to this setting.

\begin{lem}[Rational Markov]
\label{lem:rational}
Let $r=\frac{f}{g}$ be a rational function with $f,g\in\mathcal{P}_q$, 
and let $a>0$ and $m\in\mathbb{N}$. Suppose that 
$c:=\sup_{x,y\in[0,a]}\big|\frac{g(x)}{g(y)}\big|<\infty$. Then
$$
	\|r^{(m)}\|_{C^0[0,a]} \le
	m!\, \bigg(\frac{5cq^2}{a}\bigg)^m \|r\|_{C^0[0,a]}.
$$
\end{lem}

\begin{proof}
Applying the product rule to $f=rg$ yields
$$
	r^{(m)} g = 
	f^{(m)} - \sum_{k=1}^m {m\choose k} r^{(m-k)} g^{(k)}.
$$
As $\frac{1}{c}\|\frac{h}{g}\|_{C^0[0,a]}
\le \frac{\|h\|_{C^0[0,a]}}{\|g\|_{C^0[0,a]}} \le
\|\frac{h}{g}\|_{C^0[0,a]}$ for every function $h$, we can estimate
\begin{align*}
	\|r^{(m)}\|_{C^0[0,a]} &\le
	\bigg\|
	\frac{f^{(m)}}{g}\bigg\|_{C^0[0,a]}
 	+ \sum_{k=1}^m {m\choose k} \|r^{(m-k)}\|_{C^0[0,a]}
	\bigg\|\frac{g^{(k)}}{g}\bigg\|_{C^0[0,a]}
\\
	&\le
	2c\sum_{k=1}^m {m\choose k} 
	\frac{1}{(2k-1)!!} \bigg(\frac{2q^2}{a}\bigg)^k
	\|r^{(m-k)}\|_{C^0[0,a]}
\end{align*}
by applying Lemma \ref{lem:markov} to $f^{(m)}$ and $g^{(k)}$. Here
the factor $2$ in the second line is due to the fact that
the $f^{(m)}$ and $k=m$ terms in the first line yield the same 
bound.

We now reason by induction. Clearly the conclusion holds $m=0$. 
Now suppose the conclusion holds up to $m-1$. Then the above
inequality yields
\begin{align*}
	\|r^{(m)}\|_{C^0[0,a]} &\le
	\bigg(\frac{5cq^2}{a}\bigg)^m
	\|r\|_{C^0[0,a]} 
	\cdot
	2c
	\sum_{k=1}^m {m\choose k} 
	\frac{(m-k)!}{(2k-1)!!} \bigg(\frac{2}{5c}\bigg)^k
\\
	&\le
	m!\,
	\bigg(\frac{5cq^2}{a}\bigg)^m
	\|r\|_{C^0[0,a]} 
	\cdot
	2c\,
	\Big(\cosh\Big(\sqrt{\tfrac{4}{5c}}\Big)-1\Big)
\end{align*}
as ${m\choose k} \frac{(m-k)!}{(2k-1)!!} =\frac{2^k m!}{(2k)!}$.
Finally, use
$c\ge 1$ and 
$\cosh(\sqrt{x})-1 \le \frac{5x}{8}$ for 
$x\in[0,1]$.
\end{proof}

The above bounds cannot be essentially improved without further 
assumptions. Lemma~\ref{lem:markov} attains equality for 
Chebyshev polynomials \cite[p.\ 256(d)]{BE95}. The optimality of Lemma 
\ref{lem:interpol} is discussed in \cite{CR92}, while the optimality of 
Lemma \ref{lem:rational} is illustrated by considering 
$r(x)=\frac{1}{u-x}$ on $x\in[0,1]$ for $u>1$.

\subsection{Chebyshev polynomials}

Let $h\in\mathcal{P}_q$ and fix $K>0$. In this section,
we will consider the behavior of $h$ on the interval $[-K,K]$.

Denote by $T_j$ the Chebyshev polynomial of the first kind of degree $j$, 
defined by $T_j(\cos\theta) = \cos(j\theta)$. Then $h$ can be 
uniquely expressed as
\begin{equation}
\label{eq:chebexp}
        h(x) = \sum_{j=0}^q a_j T_j(K^{-1}x)
\end{equation}
for some real coefficients $a_0,\ldots,a_q$. Note that these coefficients 
are precisely the Fourier coefficients of the cosine series 
$h(K\cos\theta)$. We can therefore apply a classical result of
Zygmund on absolute convergence of trigonometric series.

\begin{lem}[Zygmund]
\label{lem:absconv}
Let $h$ be as in \eqref{eq:chebexp} and define $f(\theta) := h(K\cos\theta)$.
Then
$$
	|a_0|\le \|h\|_{C^0[-K,K]},
$$
and
$$
        \sum_{j=1}^q j^m |a_j| 
        \lesssim
	\beta_*
	\|f^{(m+1)}\|_{L^\beta[0,2\pi]}
$$
for every $m\in\mathbb{Z}_+$ and $\beta>1$, where we defined 
$1/\beta_*:=1-1/\beta$.
\end{lem}

\begin{proof}
That $|a_0|\le \|h\|_{C^0[-K,K]}$ follows immediately from
$a_0 = \frac{1}{2\pi}\int_0^{2\pi} f(\theta)\,d\theta$.
To obtain the second estimate, we note that 
$$
        f^{(m)}(\theta) = 
        \begin{cases}
        \sum_{j=0}^q (-1)^{m/2} j^m a_j \cos(j\theta)
        & \text{for even }m,\\
        \sum_{j=0}^q (-1)^{(m+1)/2} j^m a_j \sin(j\theta)
        & \text{for odd }m.
        \end{cases}
$$
The conclusion follows from \cite[p.\ 242]{Zyg02}.
\end{proof}

A direct consequence is the following.

\begin{cor}
\label{cor:abscp}
Let $h$ be as in \eqref{eq:chebexp} and let $m\in\mathbb{Z}_+$. Then
$$
        \sum_{j=0}^q j^m |a_j| 
        \le c_{m,K}\, \|h\|_{C^{m+1}[-K,K]},
$$
where the constant $c_{m,K}$ depends on $m,K$ only.
\end{cor}

\begin{proof}
That $f(\theta)=h(K\cos\theta)$ satisfies
$\|f^{(m+1)}\|_{L^\infty[0,2\pi]} \le c_{m,K} \|h\|_{C^{m+1}[-K,K]}$
follows by the chain rule. It remains to apply
Lemma \ref{lem:absconv} with $\beta=\infty$.
\end{proof}

\subsection{Compactly supported distributions}
\label{sec:schwartz}

In this paper we will encounter only distributions on $\mathbb{R}$. We
therefore adopt the following definition \cite[\S 2.3]{Hor03}.

\begin{defn}
\label{defn:distr}
A linear functional
$\nu$ on $C^\infty(\mathbb{R})$ such that 
$$
	|\nu(f)| \le c\|f\|_{C^m[-K,K]}\quad
	\text{for all }f\in C^\infty(\mathbb{R})
$$
holds for some $c,K>0$, $m\in\mathbb{Z}_+$,
is called a \emph{compactly supported distribution}.
\end{defn}

The linear functionals that will arise in this paper will not be defined 
\emph{a priori} on $C^\infty(\mathbb{R})$, but rather only on the space 
$\mathcal{P}$ of univariate polynomials. It will be therefore essential to 
understand when linear functionals on $\mathcal{P}$ can be extended to 
compactly supported distributions. As we have not located the following 
result in the literature, we prove it here for completeness.

\begin{lem}[Hausdorff moment problem for compactly supported distributions]
\label{lem:hauss}
Let $\nu$ be a linear functional on $\mathcal{P}$. Then the 
following are equivalent.
\begin{enumerate}[leftmargin=*,label=\arabic*.]
\addtolength{\itemsep}{.5mm}
\item 
There exist $c,m,K\ge 0$ so that 
$|\nu(h)| \le c q^m \|h\|_{C^0[-K,K]}$
for all $h\in\mathcal{P}_q,~q\in\mathbb{N}$.
\item
There exist $c,m,K\ge 0$ so that
$|\nu(T_j(K^{-1}\cdot))| \le c j^m$ for all $j\in\mathbb{N}$.
\item
$\nu$ extends uniquely to a compactly supported distribution.
\end{enumerate}
\end{lem}

Only the implication $1\Rightarrow 3$ will be used in this paper; the 
equivalence of $2\Leftrightarrow 3$ appears implicitly in 
\cite[p.\ 38]{EK87}.

\begin{proof}
That $1\Rightarrow 2$ is immediate as
$\|T_j(K^{-1}\cdot)\|_{C^0[-K,K]}=1$ by the definition of 
the Chebyshev polynomials $T_j$.

To prove that $2\Rightarrow 3$, fix any $h\in\mathcal{P}$ and express
it in the form \eqref{eq:chebexp}. Then 
$$
	|\nu(h)| \le
	\sum_{j=0}^q |\nu(T_j(K^{-1}\cdot))|\,|a_j| \le
	c\sum_{j=0}^q j^m|a_j|
	\le c_{m,K} \|h\|_{C^{m+1}[-K,K]}
$$
by Corollary \ref{cor:abscp}. Thus the condition of Definition 
\ref{defn:distr} is satisfied for all $h\in\mathcal{P}$.
As $\mathcal{P}$ is dense in $C^\infty(\mathbb{R})$ with respect to
the norm $\|\cdot\|_{C^{m+1}[-K,K]}$, it is clear that $\nu$ extends
uniquely to a compactly supported distribution.

To prove that $3\Rightarrow 1$, it suffices by Definition 
\ref{defn:distr} to note that $\|h\|_{C^m[-K,K]}\le 
c_{m,K} q^{2m}\|h\|_{C^0[-K,K]}$ for every 
$h\in\mathcal{P}_q$ and $q\in\mathbb{N}$ by Lemma \ref{lem:markov}.
\end{proof}

Next, we recall the definition of support \cite[\S 2.2]{Hor03}.

\begin{defn}
The \emph{support} $\supp\nu$ of a compactly supported distribution $\nu$ 
is the smallest closed set $A\subset\mathbb{R}$ with the property that 
$\nu(f)=0$ for every $f\in C^\infty(\mathbb{R})$ such that $f=0$ in a 
neighborhood of $A$.
\end{defn}

It is clear that when the condition of Definition \ref{defn:distr} is 
satisfied, we must have $\supp\nu\subseteq [-K,K]$. However, the support 
may in fact be much smaller. A key property of compactly supported 
distributions for our purposes is that their support can be bounded by the 
growth rate of their moments. While the analogous property of measures is 
straightforward, its validity for distributions requires justification.

\begin{lem}
\label{lem:distsupp}
Let $\nu$ be a compactly supported distribution. Then
$$
        \supp\nu \subseteq [-\rho,\rho] \qquad\text{with}\qquad
        \rho = \limsup_{p\to\infty} |\nu(x^p)|^{1/p}.
$$
\end{lem}

\begin{proof}
It is clear from the definition of a compactly supported distribution that 
we must have $\rho<\infty$. Thus the function $F$ defined by
$$
        F(z) := \sum_{p=0}^\infty \frac{\nu(x^p)}{z^{p+1}}
$$
is analytic for $z\in\mathbb{C}$ with $|z|$ sufficiently large. It follows
from \cite[Lemma 1]{EK87} (or from \cite[Theorem 5.4]{Sch05}) that
$F$ can be analytically continued to $\mathbb{C}\backslash\supp\nu$ and
that $\supp\nu$ is precisely the set of singular points of $F$.
As the above expansion of $F$ is absolutely convergent in a 
neighborhood of $z=\pm(\rho+\varepsilon)$ for all $\varepsilon>0$, it 
follows that $\pm(\rho+\varepsilon)\not\in\supp\nu$ for all 
$\varepsilon>0$, which yields the conclusion.
\end{proof}

\subsection{Test functions}

We finally recall a standard construction of smooth approximations of the 
indicator function $1_{[-\rho,\rho]^c}$ for which the bound of Lemma 
\ref{lem:absconv} is well controlled; its sharpness is discussed in
\cite[pp.\ 19--22]{Hor03}. Recall that $1/\beta_*=1-1/\beta$.

\begin{lem}[Test functions]
\label{lem:test}
Fix $m\in\mathbb{Z}_+$ and $K,\rho,\varepsilon>0$ so that 
$\rho+\varepsilon<K$. Then there exists a function $\chi\in 
C^\infty(\mathbb{R})$ with the following properties.
\smallskip
\begin{enumerate}[leftmargin=*,label=\arabic*.]
\addtolength{\itemsep}{1.5mm}
\item $\chi(x)\in[0,1]$ for all $x$, 
$\chi(x)=0$ for $|x|\le\rho+\frac{\varepsilon}{2}$, and $\chi(x)=1$ for
$|x|\ge\rho+\varepsilon$.
\item $\|f^{(m+1)}\|_{L^\beta[0,2\pi]} \le (Cm)^m 
(\frac{K}{\varepsilon})^{m+1/\beta_*}$
for all $\beta>1$.
\end{enumerate}
\smallskip
Here $f(\theta):=\chi(K\cos\theta)$, and $C$ is universal constant.
\end{lem}

\begin{proof}
Let $F\in C^\infty(\mathbb{R})$ be a nonnegative function that is
supported in $[0,1]$ with $\|F\|_{L^1(\mathbb{R})}=1$.
For $a>0$, define $F_a(x) := \frac{1}{a}F(\frac{x}{a})$
and $H_a(x) := \frac{1}{a}1_{[0,a]}(x)$, and let
$$
        h(x) = \int_{-\infty}^x (F_{\delta/2} *
        \underbrace{H_{\delta/2m} * \cdots * H_{\delta/2m}}_{m\text{ times}})(y)
        \,dy.
$$
Here $\delta,\varphi>0$ are parameters that will be chosen below.

Note that $h\in C^\infty(\mathbb{R})$ by construction.
Moreover, the integrand in the definition of $h$ is nonnegative, has 
$L^1(\mathbb{R})$-norm one, and is supported in $[0,\delta]$.
Thus $h$ is nondecreasing, $h(x)=0$ for $x\le 0$, and
$h(x)=1$ for $x\ge\delta$. We define
$$
        \chi(x) = 
        h\big(\arcsin\big(\tfrac{x}{K}\big)-\varphi\big) +
        h\big({-\arcsin\big(\tfrac{x}{K}\big)}-\varphi\big)
$$
for $x\in[-K,K]$, and define $\chi(x)=1$ otherwise.
We now choose the parameters 
$\varphi=\arcsin(\frac{\rho+\varepsilon/2}{K})$ and $\delta = 
\arcsin(\frac{\rho+\varepsilon}{K})-\arcsin(\frac{\rho+\varepsilon/2}{K})$ 
so that 1.\ holds. Moreover, $\chi(x)=1$ is constant near
$x=\pm K$ as $\rho+\varepsilon<K$, so $\chi\in C^\infty(\mathbb{R})$ by 
the chain rule.

Now note that
$f(\theta) = h\big(\tfrac{\pi}{2}-\theta-\varphi\big) +
h\big(\theta-\tfrac{\pi}{2}-\varphi\big)$
for $\theta\in[0,\pi]$. Therefore
\begin{align*}
        \|f^{(m+1)}\|_{L^\beta[0,2\pi]} &\le
        4\|F_{\delta/2} * H_{\delta/2m}' * 
        \cdots * H_{\delta/2m}'\|_{L^\beta(\mathbb{R})}
\\
        &\le
        4\|F_{\delta/2}\|_{L^\beta(\mathbb{R})}
        \|H_{\delta/2m}'\|_{L^1(\mathbb{R})}^m
        \le 
        (4m)^m \big(\tfrac{1}{\delta}\big)^{m+1/\beta_*}
	8\|F\|_{L^\infty(\mathbb{R})},
\end{align*}
where the factor $4$ in the first line arises by the triangle
inequality and by splitting the $L^\beta$ norm over the intervals
$[0,\pi]$ and $[\pi,2\pi]$, and the second line uses
Young's inequality and $H_a' = \frac{1}{a}(\delta_0-\delta_a)$.
Then 2.\ follows by noting that $\delta\ge \frac{\varepsilon}{2K}$.
\end{proof}

\section{Words in random permutation matrices}
\label{sec:words}

The aim of this section is to recall basic facts about random permutations 
that will form the input to the methods of this paper. These results are 
implicit in Nica \cite{Nic94}, but are derived in simpler and more 
explicit form by Linial and Puder \cite{PL10} whose presentation we 
follow. We emphasize that all the results that are reviewed in this 
section arise from elementary combinatorial reasoning; we refer the reader 
to the expository paper \cite[\S 3]{Col22} for a brief introduction.

We will work with random permutation 
matrices $\boldsymbol{S}^N=(S_1^N,\ldots,S_d^N)$ and their limiting model 
$\boldsymbol{s}=(s_1,\ldots,s_d)$ as defined in section \ref{sec:prelim}.
In particular, we recall that $s_i:=\lambda(g_i)$, where
$g_1,\ldots,g_d$ are the free generators of $\mathbf{F}_d$.

It will be convenient to extend these definitions by setting $g_0:=e$, 
$g_{d+i}:=g_i^{-1}$ and analogously $S^N_0:=\id$, $S^N_{d+i}:=(S_i^N)^{-1} 
= S_i^{N*}$ and $s_0:=\id$, $s_{d+i}:=s_i^{-1}=s_i^*$ for $1\le i\le d$. 
This convention will be in force in the remainder of the paper.

We aim to understand the expected trace of words in random permutation 
matrices and their inverses. To formalize this notion, denote by 
$\mathbf{W}_d$ the set of all finite words in the letters 
$g_0,\ldots,g_{2d}$. We implicitly identify $w\in\mathbf{W}_d$ with 
the word map $w:\mathrm{G}^d\to\mathrm{G}$ it induces on any group 
$\mathrm{G}$. Thus $w(g_1,\ldots,g_d)\in\mathbf{F}_d$ is the reduction of 
the word $w$, while $w(S_1^N,\ldots,S_d^N)$ is the random matrix obtained 
by substituting $g_i\leftarrow S_i^N$ and multiplying these matrices in 
the order they appear in $w$.

\subsection{Rationality}

The first property we will need is that the expected trace of any word is 
a rational function of $\frac{1}{N}$. We emphasize that only very limited 
information on this function will be needed, which is collected in the 
following lemma.

\begin{lem}
\label{lem:wordrational}
Let $w\in\mathbf{W}_d$ be any word of length at most $q$ and $N\ge q$. 
Then 
$$
        \mathbf{E}[\ntrN w(\boldsymbol{S}^N)] =
	\frac{f_w(\tfrac{1}{N})}{g_q(\tfrac{1}{N})},
$$
where
$$
	g_q(x) :=
	(1-x)^{d_1}(1-2x)^{d_2}\cdots 
	(1-(q-1)x)^{d_{q-1}}
$$
with $d_j := \min\big(d,\lfloor \frac{q}{j+1}\rfloor\big)$, 
and $f_w,g_q$ are polynomials of degree at most $q(1+\log d)$.
\end{lem}

\begin{proof}
Observe that $\mathbf{E}[\tr w(S_1^N,\ldots,S_d^N)]$ is the expected number
of fixed points of $w(\bar S_1^N,\ldots,\bar S_d^N)$ minus one (as we
restrict to $\{1_N\}^\perp$). Thus \cite[eq.\ (7)]{PL10} yields
$$
        \mathbf{E}[\ntrN w(\boldsymbol{S}^N)] =
	-\frac{1}{N} +
        \sum_\Gamma
        \frac{
        (\frac{1}{N})^{e_\Gamma-v_\Gamma+1}
        \prod_{l=1}^{v_\Gamma-1}(1-\frac{l}{N})}
        {\prod_{j=1}^d \prod_{l=1}^{e_\Gamma^j-1} (1-\frac{l}{N})}
$$
where the sum is over a certain collection of connected graphs $\Gamma$ 
with $v_\Gamma \le q$ vertices and $e_\Gamma\le q$ edges, and 
$e_\Gamma^j\ge 0$ are integers so that 
$e_\Gamma^1+\cdots+e_\Gamma^d=e_\Gamma$ \cite[p.\ 105]{PL10}. Note that 
$e_\Gamma - v_\Gamma+1\ge 0$ as $\Gamma$ is connected.

The denominator inside the sum can be written equivalently as
$$
	{\textstyle
        \prod_{j=1}^d \prod_{l=1}^{e_\Gamma^j-1} (1-\tfrac{l}{N})
        =
        (1-\tfrac{1}{N})^{d_1^\Gamma} 
        (1-\tfrac{2}{N})^{d_2^\Gamma} \cdots
        (1-\tfrac{q-1}{N})^{d_{q-1}^\Gamma}
	},
$$
with $d_i^\Gamma := |\{1\le j\le d: e_\Gamma^j \ge i+1\}|\le d_i$ as
$e_\Gamma^1+\cdots+e_\Gamma^d\le q$. Thus
$$
        \mathbf{E}[\ntrN w(\boldsymbol{S}^N)] =
	\frac{
	-\frac{1}{N}g_q(\frac{1}{N})
        +\sum_\Gamma
        (\frac{1}{N})^{e_\Gamma-v_\Gamma+1}
        \prod_{l=1}^{v_\Gamma-1}(1-\frac{l}{N})
	\prod_{i=1}^{q-1}
	(1-\frac{i}{N})^{d_i-d_i^\Gamma}
	}{
	g_q(\frac{1}{N})}.
$$
To conclude, note that $\sum_{i=1}^{q-1} 
d_i \le \int_1^q \min(d,\frac{q}{x}) dx \le q(1+\log d) - \min(d,q)$
and $e_\Gamma - \sum_{i=1}^{q-1} d_i^\Gamma =
|\{1\le j\le d:e_\Gamma^j\ge 1\}| \le \min(d,q)$. Thus $g_q$ has degree
at most $q(1+\log d)-1$ and each term in the sum has degree at most
$q(1+\log d)$.
\end{proof}

\subsection{First-order asymptotics}
\label{sec:lowasymp}

We now recall the first-order asymptotic behavior of expected traces of 
words in random permutation matrices. The following is a special case of a 
result of Nica \cite{Nic94}, see \cite[p.\ 124]{PL10} for a simple proof.

An element of the free group $v\in\mathbf{F}_d$, $v\ne e$ is 
called a \emph{proper power} if it can be written as $v=w^k$ for some 
$w\in\mathbf{F}_d$ and $k\ge 2$, and is called a \emph{non-power} otherwise.
We denote by $\mathbf{F}_d^{\rm np}$ the set of non-powers. 
Every $v\in\mathbf{F}_d$, $v\ne e$ can be 
written uniquely as $v=w^k$ for some $w\in\mathbf{F}_d^{\rm np}$ and $k\ge 
1$, cf.\ \cite[Proposition I.2.17]{LS01}.

\begin{lem}
\label{lem:nica}
Fix a word $w\in\mathbf{W}_d$ that does not reduce to the identity, and express 
its reduction as $w(g_1,\ldots,g_d)=v^k$ with
$v\in\mathbf{F}_d^{\rm np}$ and $k\ge 1$. Then
$$
        \lim_{N\to\infty}
        N\,\mathbf{E}[\ntrN w(\boldsymbol{S}^N)] = \omega(k)-1,
$$
where $\omega(k)$ denotes the number of divisors of $k$.
\end{lem}

\begin{proof}
This follows from \cite[Corollary 1.3]{Nic94} by noting
as in the proof of Lemma \ref{lem:wordrational} that
$\mathbf{E}[\tr w(\boldsymbol{S}^N)]$ is the expected number of fixed points
of $w(\boldsymbol{\bar S}^N)$ minus one.
\end{proof}

In particular, Lemma \ref{lem:nica} implies the following.

\begin{cor}[Weak convergence]
\label{cor:waf}
For every word $w\in\mathbf{W}_d$, we have
$$
	\lim_{N\to\infty} 
	\mathbf{E}[\ntrN w(\boldsymbol{S}^N)] =
	\tau\big(w(\boldsymbol{s})\big).
$$
\end{cor}

\begin{proof}
Lemma \ref{lem:nica} implies that $\EE[\ntrN w(\boldsymbol{S}^N)]=o(1)$ 
for any word $w$ that does not reduce to the identity. On the other hand, 
if $w$ reduces to the identity, then $w(\boldsymbol{S}^N)$ is the identity 
matrix on $\{1_N\}^\perp$ and thus $\EE[\ntrN 
w(\boldsymbol{S}^N)]=1-\frac{1}{N}$. The conclusion follows as clearly 
$\tau(s_{i_1}\cdots s_{i_k})=
\langle \delta_e,\lambda(g_{i_1})\cdots\lambda(g_{i_k}) 
\delta_e\rangle=1_{g_{i_1}\cdots g_{i_k}=e}$.
\end{proof}

\section{Master inequalities for random permutations I}
\label{sec:masterI}

\subsection{Master inequalities}

In this section, we develop a core ingredient of our method: 
inequalities for the normalized trace of polynomials of 
$\boldsymbol{S}^N$.

\begin{thm}[Master inequalities]
\label{thm:masterperm}
Fix a self-adjoint noncommutative polynomial $P\in
\mathrm{M}_D(\mathbb{C})\otimes\mathbb{C}\langle
\boldsymbol{s},\boldsymbol{s}^*\rangle$ of degree $q_0$, and let 
$K=\|P\|_{\mathrm{M}_D(\mathbb{C})\otimes C^*(\mathbf{F}_d)}$. Then there 
exists a linear functional $\nu_i$ on $\mathcal{P}$ for every 
$i\in\mathbb{Z}_+$ so that
$$
	\bigg|\EE[
	\ntrND h(P(\boldsymbol{S}^N,\boldsymbol{S}^{N*}))]
	- \sum_{i=0}^{m-1} \frac{\nu_i(h)}{N^i}
	\bigg|
	\le
	\frac{(4qq_0(1+\log d))^{4m}}{N^m} \|h\|_{C^0[-K,K]}
$$
for every $N,m,q\in\mathbb{N}$ and $h\in\mathcal{P}_q$.
\end{thm}

An immediate consequence of Theorem \ref{thm:masterperm} is the following 
corollary that we spell out separately for future reference.

\begin{cor}
\label{cor:linfunbd}
In the setting of Theorem \ref{thm:masterperm}, we have
$$
	|\nu_m(h)| \le (4qq_0(1+\log d))^{4m}\,\|h\|_{C^0[-K,K]}
$$
for every $m\in\mathbb{Z}_+$, $q\in\mathbb{N}$, and $h\in\mathcal{P}_q$.
\end{cor}

While we will need Theorem \ref{thm:masterperm} in its full generality in 
some applications, many strong convergence results will already arise from 
the case $m=2$ as was explained in section \ref{sec:outline}. The case 
$m=1$ is of independent interest as 
it provides a quantitative form of Corollary \ref{cor:waf}; see, for
example, Corollary \ref{cor:quantwaf} below.

\subsection{Proof of Theorem \ref{thm:masterperm} and
Corollary \ref{cor:linfunbd}}

Throughout the proof, we fix a polynomial $P$ as in Theorem 
\ref{thm:masterperm}. We emphasize that all the objects that appear in the 
proof depend implicitly on the choice of $P$.

We begin by noting that
$({\ntrD}\otimes\mathrm{id})(h(P(\boldsymbol{S}^N,\boldsymbol{S}^{N*})))$ 
is a linear combination of words $w(\boldsymbol{S}^N)$ of length at most 
$qq_0$ for every $h\in\mathcal{P}_q$.
Thus Lemma \ref{lem:wordrational} yields
$$
	\EE[\ntrND h(P(\boldsymbol{S}^N,\boldsymbol{S}^{N*}))]
	=
	\Psi_h(\tfrac{1}{N}) = \frac{f_h(\frac{1}{N})}{g_{qq_0}(\frac{1}{N})}
$$
for $N\ge qq_0$,
where $f_h,g_{qq_0}$ are polynomials of degree at most $qq_0(1+\log d)$.

As $\Psi_h$ has no pole at $0$, we can define for each 
$m\in\mathbb{Z}_+$
$$
	\nu_m(h) := \frac{\Psi_h^{(m)}(0)}{m!}.
$$
It is clear from the definition of $\Psi_h$ that $\nu_m$ defines a linear 
functional on $\mathcal{P}$. Moreover, it follows immediately from 
Taylor's theorem that
\begin{equation}
\label{eq:taylor}
	\bigg|\EE[\ntrND h(P(\boldsymbol{S}^N,\boldsymbol{S}^{N*}))]
	-
	\sum_{i=0}^{m-1} \frac{\nu_i(h)}{N^i}
	\bigg|
	\le
	\frac{\|\Psi_h^{(m)}\|_{C^0[0,\frac{1}{N}]}}{m!}
	\frac{1}{N^m}
\end{equation}
provided $N$ is large enough that $\Psi_h$ has no poles on 
$[0,\frac{1}{N}]$.
The main idea of the proof is that we will use Lemma \ref{lem:rational} to 
bound the remainder term in the Taylor expansion. To this end we must show
that the function $g_{qq_0}$ is nearly constant.

\begin{lem}
\label{lem:gconst}
$e^{-1} \le g_{q}(x)\le 1$ for every
$x\in[0,\frac{1}{q^2(1+\log d)}]$ and $q\in\mathbb{N}$.
\end{lem}

\begin{proof}
The definition of $g_q$ in Lemma \ref{lem:wordrational} implies 
$(1-qx)^{q(1+\log d)-1} \le g_q(x) \le 1$ for $x\in[0,\frac{1}{q}]$, where 
we used that $d_1+\cdots+d_{q-1}\le q(1+\log d)-1$ as shown in the proof 
of Lemma \ref{lem:wordrational}. The conclusion follows using 
$(1-\frac{1}{a})^{a-1} \ge e^{-1}$ for $a>1$.
\end{proof}

Next, we must obtain an upper bound on $\|\Psi_h\|_{C^0[0,\frac{1}{N}]}$.

\begin{lem}
\label{lem:psihbd}
$\|\Psi_h\|_{C^0[0,\frac{1}{M}]}\le 2e\|h\|_{C^0[-K,K]}$
for $M=2(qq_0(1+\log d))^2$.
\end{lem}

\begin{proof}
As $\|P(\boldsymbol{S}^N,\boldsymbol{S}^{N*})\| \le 
\|P\|_{\mathrm{M}_D(\mathbb{C})\otimes C^*(\mathbf{F}_d)}$, we have
$$
	|\Psi_h(\tfrac{1}{N})| =
	|\EE[\ntrND h(P(\boldsymbol{S}^N,\boldsymbol{S}^{N*}))]|
	\le \|h\|_{C^0[-K,K]}	
$$
for all $N\in\mathbb{N}$. We will extend this bound from the
set $I=\{\frac{1}{N}:N\ge M\}$ to the interval $[0,\frac{1}{M}]$
using polynomial interpolation.
First, note that 
$$
	|f_h(\tfrac{1}{N})| \le |\Psi_h(\tfrac{1}{N})| \le \|h\|_{C^0[-K,K]}
$$
as $g_{qq_0}(\frac{1}{N})\le 1$ for $N\ge q$.
The assumption of Lemma \ref{lem:interpol}
is satisfied for $h\leftarrow f_h$, $q\leftarrow qq_0(1+\log d)$, 
$a=\frac{1}{M}$ as soon as $M\ge 2(qq_0(1+\log d))^2$, which 
yields
$$
	\|f_h\|_{C^0[0,\frac{1}{M}]} \le
	2\|h\|_{C^0[-K,K]}.
$$
It remains to apply
Lemma \ref{lem:gconst} to estimate $\|\Psi_h\|_{C^0[0,\frac{1}{M}]}
\le e\|f_h\|_{C^0[0,\frac{1}{M}]}$.
\end{proof}

We can now conclude the proof.

\begin{proof}[Proof of Theorem \ref{thm:masterperm} and Corollary 
\ref{cor:linfunbd}]
Let $M= 2(qq_0(1+\log d))^2$. Then
$$
	\frac{\|\Psi_h^{(m)}\|_{C^0[0,\frac{1}{M}]}}{m!}
	\le (4qq_0(1+\log d))^{4m} 
	\|h\|_{C^0[-K,K]}
$$
by Lemmas \ref{lem:rational}, \ref{lem:gconst}, and \ref{lem:psihbd},
where we used that $(10e)^m 2e \le 4^{4m}$ for $m\ge 1$. Thus the proof of 
Theorem \ref{thm:masterperm} for $N\ge M$ follows directly from
\eqref{eq:taylor}. 

The proof of Corollary \ref{cor:linfunbd} for $m\ge 1$ now follows by 
multiplying the bound of Theorem \ref{thm:masterperm} by $N^m$ and taking 
$N\to\infty$. As $\nu_0(h)=\lim_{N} \EE[ \ntrND 
h(P(\boldsymbol{S}^N,\boldsymbol{S}^{N*}))]$, the conclusion of 
Corollary \ref{cor:linfunbd} evidently also holds for $m=0$.

It remains to prove Theorem \ref{thm:masterperm} for $N<M$. To this end,
we can trivially bound the left-hand side of the inequality in
Theorem \ref{thm:masterperm} by
$$
	\bigg|\EE[
	\ntrND h(P(\boldsymbol{S}^N,\boldsymbol{S}^{N*}))]
	- \sum_{i=0}^{m-1} \frac{\nu_i(h)}{N^i}
	\bigg|
	\le
	\bigg(
	1
	+ \sum_{i=0}^{m-1} \frac{(4qq_0(1+\log d))^{4i}}{N^i}
	\bigg)
	\|h\|_{C^0[-K,K]}
$$
using the triangle inequality and Corollary~\ref{cor:linfunbd}.
But note that $1<\frac{2^k(qq_0(1+\log d))^{4k}}{N^k}$ for 
every $k\ge 1$ as we assumed $N<M$. We can therefore estimate
$$
	1
	+ \sum_{i=0}^{m-1} \frac{(4qq_0(1+\log d))^{4i}}{N^i}
	\le
	\bigg(
	2^m + \sum_{i=0}^{m-1} 4^{4i} 2^{m-i}
	\bigg)
	\frac{(qq_0(1+\log d))^{4m}}{N^m},
$$
and we conclude using
$2^m + \sum_{i=0}^{m-1} 4^{4i} 2^{m-i}\le 4^{4m}$.
\end{proof}

\section{Master inequalities for random permutations II}
\label{sec:masterII}

The aim of this short section is to extend the master inequalities of 
Theorem~\ref{thm:masterperm} from polynomials to general smooth test 
functions. This extension will play a key role in the proofs of strong 
convergence.

\begin{thm}[Smooth master inequalities]
\label{thm:smmasterperm}
Fix a self-adjoint noncommutative polynomial $P\in
\mathrm{M}_D(\mathbb{C})\otimes\mathbb{C}\langle
\boldsymbol{s},\boldsymbol{s}^*\rangle$ of degree $q_0$, and let
$K=\|P\|_{\mathrm{M}_D(\mathbb{C})\otimes C^*(\mathbf{F}_d)}$.
Then there exists a compactly supported
distribution $\nu_i$ for every $i\in\mathbb{Z}_+$ so that
$$
	\bigg|\EE[
	\ntrND h(P(\boldsymbol{S}^N,\boldsymbol{S}^{N*}))]
	- \sum_{i=0}^{m-1} \frac{\nu_i(h)}{N^i}
	\bigg|
	\lesssim
	\frac{(4q_0(1+\log d))^{4m}}{N^m} 
	\,\beta_*\|f^{(4m+1)}\|_{L^\beta[0,2\pi]}
$$
for all $N,m\in\mathbb{N}$, $\beta>1$, $h\in C^\infty(\mathbb{R})$,
where $f(\theta):=h(K\cos\theta)$ and $1/\beta_*:=1-1/\beta$.
\end{thm}

\begin{proof}
We first note that Corollary \ref{cor:linfunbd} ensures, by 
Lemma \ref{lem:hauss}, that the linear functionals $\nu_i$ 
in Theorem \ref{thm:masterperm} extend uniquely to compactly supported
distributions.

Let $h\in\mathcal{P}_q$ and express it in the form \eqref{eq:chebexp}.
Rather than applying Theorem \ref{thm:masterperm} to $h$ directly, we apply
it to the Chebyshev polynomials $T_j(K^{-1}\cdot)$ to obtain
$$
	\bigg|\EE[
	\ntrND h(P(\boldsymbol{S}^N,\boldsymbol{S}^{N*}))]
	- \sum_{i=0}^{m-1} \frac{\nu_i(h)}{N^i}
	\bigg|
	\le
	\frac{(4q_0(1+\log d))^{4m}}{N^m}
	\sum_{j=1}^q
	j^{4m}
	|a_j|.
$$
Here we used that the left-hand side of Theorem \ref{thm:masterperm} 
vanishes when $h$ is the constant function (as then $\nu_0(h)=h$, 
$\nu_i(h)=0$ for $i\ge 1$), so the constant term in the Chebyshev 
expansion cancels. The proof is completed for
$h\in\mathcal{P}_q$ by applying Lemma \ref{lem:absconv}.
The conclusion extends to $h\in C^\infty$
as $\mathcal{P}_q$ is dense in $C^\infty[-K,K]$.
\end{proof}

The special case $m=1$ is of independent interest, as it yields a 
quantitative formulation of weak convergence (cf.\ Corollary 
\ref{cor:waf}).

\begin{cor}[Quantitative weak convergence]
\label{cor:quantwaf}
In the setting of Theorem \ref{thm:smmasterperm},
$$
        \big|\EE[\ntrND h(P(\boldsymbol{S}^N,\boldsymbol{S}^{N*}))] -
        ({\ntrD}\otimes{\tau})(h(P(\boldsymbol{s},\boldsymbol{s}^*)))
	\big|
        {\,\lesssim\,}
        \frac{(q_0(1+\log d))^{4}}{N}
        \beta_*\|f^{(5)}\|_{L^\beta[0,2\pi]}
$$
for all $N\in\mathbb{N}$, $\beta>1$, $h\in C^\infty(\mathbb{R})$,
where $f(\theta):=h(K\cos\theta)$ and $1/\beta_*:=1-1/\beta$.
\end{cor}

\begin{proof}
This is simply a restatement of the special case of
Theorem \ref{thm:smmasterperm} with $m=1$, where we note that
$\nu_0(h)=({\ntrD}\otimes{\tau})(h(P(\boldsymbol{s},\boldsymbol{s}^*)))$ 
by Corollary \ref{cor:waf}.
\end{proof}

\begin{rem}
When $h$ is chosen to be a smooth approximation of the indicator function 
of an interval, Corollary \ref{cor:quantwaf} shows that the spectral 
distribution of $P(\boldsymbol{S}^N,\boldsymbol{S}^{N*})$ is well 
approximated by that of $P(\boldsymbol{s},\boldsymbol{s}^*)$ at a 
mesoscopic scale. While far sharper results at the local scale are proved 
in \cite{HY24,HMY24,HMY25} for random regular graphs,
no result of this kind appears to be known to date for arbitrary 
polynomials $P$.
\end{rem}

In the remainder of this paper, we adopt the notation for 
the distributions $\nu_i$ as in Theorem \ref{thm:smmasterperm}. It should 
be emphasized, however, that the definition of $\nu_i$ depends both on the 
model and on the noncommutative polynomial $P$ under consideration. As 
each of the following sections will be concerned with a specific model and 
choice of polynomial $P$, the meaning of $\nu_i$ will always be clear from 
context.

\section{Random regular graphs}
\label{sec:rreg}

The aim of this section is to prove our main results on random regular 
graphs, the effective Friedman Theorem \ref{thm:friedman} and the 
staircase Theorem \ref{thm:puder}.

\subsection{Proof of Theorem \ref{thm:friedman}}
\label{sec:friedman}

The proof of Theorem \ref{thm:friedman} is based on Theorem 
\ref{thm:smmasterperm} with $m=2$. Let us spell out what this says in the 
present setting.

\begin{lem}
\label{lem:masterfriedman}
There exists a compactly supported distribution $\nu_1$ such that
for every $N\in\mathbb{N}$, $\beta>1$, and $h\in C^\infty(\mathbb{R})$, we 
have
$$
	\big| \EE[\ntrN h(A^N)] - \tau(h(A_{\rm F})) - \tfrac{1}{N}\nu_1(h)
	\big|
	\lesssim \frac{(1+\log d)^8}{N^2}\, \beta_* \|f^{(9)}\|_{L^\beta[0,2\pi]},
$$
where $f(\theta) := h(2d\cos\theta)$ and $1/\beta_*=1-1/\beta$.
\end{lem}

\begin{proof}
Apply Theorem \ref{thm:smmasterperm} with $m=2$ and
$P(\boldsymbol{s},\boldsymbol{s}^*) =
s_1+s_1^*+\cdots+s_d+s_d^*$, using
$\|P\|_{C^*(\mathbf{F}_d)}=2d$ and $\nu_0(h) =
\lim_N \EE[\ntrN h(A^N)]= \tau(h(A_{\rm F}))$ by Corollary
\ref{cor:waf}.
\end{proof}

The remaining ingredient of the proof is to show that $\supp\nu_1
\subseteq [-\|A_{\rm F}\|,\|A_{\rm F}\|]$. 
To this end, is suffices by Lemma \ref{lem:distsupp} to understand the 
exponential growth rate of the moments of $\nu_1$. 
We begin by computing a formula for the moments of $\nu_1$.\footnote{%
	In the special setting of random regular graphs, a nonrigorous 
	derivation 
	of an explicit formula for $\nu_1$ (that is, an explicit solution 
	to the Hausdorff moment problem associated to Lemma \ref{lem:friedmom})
	appears in the physics literature \cite{MPL14}.
	However, an explicit solution to the moment
	problem, which is not available in more general situations,
	is not needed to apply the methods of this paper.
}
Recall that $\mathbf{F}_d^{\rm np}$ denotes the non-power elements of
$\mathbf{F}_d$ (see Lemma \ref{lem:nica}).

\begin{lem}
\label{lem:friedmom}
For all $p\in\mathbb{N}$, we have
$$
	\nu_1(x^p) = -\tau(A_{\rm F}^p) +
	\sum_{k=2}^p
	(\omega(k)-1)
	\sum_{v\in\mathbf{F}_d^{\rm np}}
	\sum_{i_1,\ldots,i_p=1}^{2d}
	1_{g_{i_1}\cdots g_{i_p} = v^k}.
$$
\end{lem}

\begin{proof}
Note that
\begin{multline*}
	\nu_1(x^p) + \tau(A_{\rm F}^p) = 
	\lim_{N\to\infty}
	N\big(\EE[\ntrN (A^N)^p] - 
	(1-\tfrac{1}{N})\,\tau(A_{\rm F}^p)\big)
\\
	=
	\lim_{N\to\infty}
	\sum_{i_1,\ldots,i_p=1}^{2d}
	N\big(\EE[\ntrN S^N_{i_1}\cdots S^N_{i_p}] -
	(1-\tfrac{1}{N})\,
	\tau(\lambda(g_{i_1})\cdots \lambda(g_{i_p}))\big).
\end{multline*}
As $\tau(\lambda(g_{i_1})\cdots \lambda(g_{i_p})) =
1_{g_{i_1}\cdots g_{i_p}=e}$, the terms with $g_{i_1}\cdots g_{i_p}=e$
cancel in the sum as in the proof of Corollary \ref{cor:waf}. The
conclusion now follows from Lemma \ref{lem:nica} (note that the $k=1$ term
does not appear in the sum as $\omega(1)-1=0$).
\end{proof}

In the present setting, a simple argument due to Friedman \cite[Lemma 
2.4]{Fri03} suffices to bound the growth rate of the moments.
As a variant of this argument will be needed later in this paper, we 
recall the proof here.

\begin{lem}
\label{lem:friedsupp}
For every $k\ge 2$ and $p\ge 1$, we have
$$
	\sum_{v\in\mathbf{F}_d^{\rm np}}
        \sum_{i_1,\ldots,i_p=1}^{2d}   
        1_{g_{i_1}\cdots g_{i_p} = v^k}
	\le (p+1)^4 \|A_{\rm F}\|^p.
$$
\end{lem}

\begin{proof}
We would like to argue that if
$g_{i_1}\cdots g_{i_p}=v^k$, there must exist $a<b$
so that $g_{i_1}\cdots g_{i_a}=v$, $g_{i_{a+1}}\cdots g_{i_b}=v$,
and $g_{i_{b+1}}\cdots g_{i_p}=v^{k-2}$. But this need not be true:
if $v$ is not cyclically reduced, the last letters of $v$
may cancel the first letters of the next repetition of $v$, and then
the cancelled letters need not appear in $g_{i_1}\cdots g_{i_p}$.

To eliminate this ambiguity, we note that any $v\in\mathbf{F}_d^{\rm np}$ 
can be uniquely expressed as $v=gwg^{-1}$ with $w\in\mathbf{F}_d^{\rm np}$,
$g\in\mathbf{F}_d$ such that $w$ is cyclically reduced and
such that $gwg^{-1}$ is reduced if $g\ne e$. Thus we may write for any $k\ge 2$
\begin{multline*}
	\sum_{v\in\mathbf{F}_d^{\rm np}}
	\sum_{i_1,\ldots,i_p=1}^{2d}
	1_{g_{i_1}\cdots g_{i_p} = v^k}
	\le
	\sum_{(w,g)}
	\sum_{0\le t_1\le t_2\le t_3\le t_4\le p}
	\sum_{i_1,\ldots,i_p=1}^{2d}
	\big(
	1_{g_{i_1}\cdots g_{i_{t_1}}=g} 
\times \mbox{}
\\
	1_{g_{i_{t_1+1}}\cdots g_{i_{t_2}}=w}\,
	1_{g_{i_{t_2+1}}\cdots g_{i_{t_3}}=w}\,
	1_{g_{i_{t_3+1}}\cdots g_{i_{t_4}}=w^{k-2}}\,
	1_{g_{i_{t_4+1}}\cdots g_{i_{p}}=g^{-1}}
	\big),
\end{multline*}
where the sum is over pairs $(w,g)$ with the above properties.

The idea is now to express the indicators as
$1_{g_{i_1}\cdots g_{i_t}=v} = 
\langle \delta_v,\lambda(g_{i_1})\cdots\lambda(g_{i_t})\delta_e\rangle$.
Substituting this identity into the right-hand side of the above
inequality yields
\begin{align*}
	&\sum_{v\in\mathbf{F}_d^{\rm np}}
	\sum_{i_1,\ldots,i_p=1}^{2d}
	1_{g_{i_1}\cdots g_{i_p} = v^k}
	\le 
	\sum_{(w,g)}
	\sum_{0\le t_1\le t_2\le t_3\le t_4\le p}
	\big(
	\langle \delta_g,A_{\rm F}^{t_1}\delta_e\rangle
	\langle \delta_w,A_{\rm F}^{t_2-t_1}\delta_e\rangle
	\times\mbox{}\\
	&
\hspace*{5cm}
	\langle \delta_w,A_{\rm F}^{t_3-t_2}\delta_e\rangle
	\langle \delta_{w^{k-2}},A_{\rm F}^{t_4-t_3}\delta_e\rangle
	\langle \delta_{g^{-1}},A_{\rm F}^{p-t_4}\delta_e\rangle
	\big)
	\\
&
	\le
	\sum_{0\le t_1\le t_2\le t_3\le t_4\le p}
	\|A_{\rm F}\|^{t_4-t_3}
	\bigg(
	\sum_{g\in\mathbf{F}_d}
	\langle \delta_g,A_{\rm F}^{t_1}\delta_e\rangle
	\langle \delta_{g^{-1}},A_{\rm F}^{p-t_4}\delta_e\rangle
	\bigg) \times\mbox{} \\
&
\hspace*{6cm}
	\bigg(
	\sum_{w\in\mathbf{F}_d}
	\langle \delta_w,A_{\rm F}^{t_2-t_1}\delta_e\rangle
	\langle \delta_w,A_{\rm F}^{t_3-t_2}\delta_e\rangle
	\bigg)
\end{align*}
where we used $|\langle \delta_{w^{k-2}},A_{\rm 
F}^{t_4-t_3}\delta_e\rangle|
\le \|A_{\rm F}\|^{t_4-t_3}$.
The conclusion now follows readily by applying
the Cauchy--Schwarz inequality to the sums over $g$ and $w$,
as
$$
	\sum_{v\in\mathbf{F}_d}
	|\langle \delta_v,A_{\rm F}^t\delta_e\rangle|^2
	=\|A_{\rm F}^t\delta_e\|^2 \le \|A_{\rm F}\|^{2t}
$$
for any $t\ge 0$.
\end{proof}


\begin{cor}
\label{cor:friedsupp}
$\supp\nu_1 \subseteq [-\|A_{\rm F}\|,\|A_{\rm F}\|]$.
\end{cor}

\begin{proof}
Lemmas \ref{lem:friedmom} and \ref{lem:friedsupp} imply that
$|\nu_1(x^p)| \le (1+p^2(p+1)^4)\|A_{\rm F}\|^p$ for all $p\ge 1$,
so that the conclusion follows from Lemma \ref{lem:distsupp}.
\end{proof}

We can now complete the proof of Theorem \ref{thm:friedman}.

\begin{proof}[Proof of Theorem \ref{thm:friedman}]
Let $\chi$ be the test function provided by Lemma \ref{lem:test} with 
$m=8$, $K=2d$, $\rho=\|A_{\rm F}\|=2\sqrt{2d-1}$. As $\chi(x)=0$ for 
$|x|\le \|A_{\rm F}\|+\frac{\varepsilon}{2}$, we have $\tau(\chi(A_{\rm 
F}))=0$ and $\nu_1(\chi)=0$ by Corollary \ref{cor:friedsupp}.
Thus Lemmas \ref{lem:masterfriedman} and \ref{lem:test} yield
$$
	\EE[\ntrN \chi(A^N)] 
	\lesssim 
	\frac{(\log d)^8}{N^2}\, \beta_* 
	\bigg(\frac{d}{\varepsilon}\bigg)^{8+1/\beta_*}
$$
for all $\varepsilon < 2d-2\sqrt{2d-1}$ and $\beta_*\ge 1$.

To conclude, we note that as 
$\chi(x)\ge 1$ for $|x|\ge \|A_{\rm F}\|+\varepsilon$, we have
$$
	\mathbf{P}[\|A^N\|\ge\|A_{\rm F}\|+\varepsilon]
	\le
	\mathbf{P}[\tr \chi(A^N) \ge 1]
	\lesssim
	\frac{1}{N}\, 
	\bigg(\frac{d\log d}{\varepsilon}\bigg)^{8}
	\log\bigg(\frac{2ed}{\varepsilon}\bigg)
$$
for $\varepsilon < 2d-2\sqrt{2d-1}$, where we chose
$\beta_*=1+\log(\frac{2d}{\varepsilon})$.
\end{proof}

\subsection{Proof of Theorem \ref{thm:puder}}
\label{sec:puder}

The proof is based on results of Puder \cite{Pud15}. Let 
us begin by synthesizing the parts of \cite{Pud15} that we need here.

\begin{lem}
\label{lem:puder}
Define $\rho_m$ as in Theorem \ref{thm:puder}. Then
for any $2\le m\le d$, we have
$$
	\limsup_{p\to\infty}|\nu_m(x^p)|^{1/p} \le
	\rho_m.
$$
\end{lem}

\begin{proof}
Fix a word $w\in\mathbf{W}_d$ of length $q$ that does not reduce to the 
identity. By Lemmas \ref{lem:wordrational} and \ref{lem:nica}, we can write 
for all sufficiently large $N$
$$
	N\,\EE[\ntrN w(\boldsymbol{S}^N)] =
	\sum_{s=0}^\infty \frac{b_s(w)}{N^s},
$$
as a rational function is analytic away from its poles.
It is shown in \cite[\S 5.5]{Pud15} that
$b_s(w)=0$ for $s\le \pi(w)-2$, that 
$b_{\pi(w)-1}(w)=|\mathrm{Crit}(w)|$, and that $|b_s(w)|\le q^{2(s+1)}$
for $s\ge\pi(w)$, where $\pi(w)\in\{0,\ldots,d\}\cup\{\infty\}$ denotes
the primitivity rank of $w$. We refer to \cite[\S 2.2]{Pud15} for
precise definitions of $\pi(w)$ and $\mathrm{Crit}(w)$, which 
are not needed in the following argument: the only properties we 
will use are that $\pi(w)=0$ if and only if $w=e$, and that
$|\mathrm{Crit}(w)|\ge 1$ when $\pi(w)\ne\infty$.

Theorem \ref{thm:masterperm} with $m=d+1$ and
$P(\boldsymbol{s},\boldsymbol{s}^*) = s_1+s_1^*+\cdots+s_d+s_d^*$ yields
\begin{align*}
	\sum_{i_1,\ldots,i_p=1}^{2d}
	N\,\EE[\ntrN S^N_{i_1}\cdots S^N_{i_p}]
	\,1_{g_{i_1}\cdots g_{i_p}\ne e} &=
	N\big(\EE[\ntrN (A^N)^p] - (1-\tfrac{1}{N})\nu_0(x^p)\big)
	\\ &=
	\nu_0(x^p) +
	\sum_{i=0}^{d-1} \frac{\nu_{i+1}(x^p)}{N^i} +
	O\bigg(\frac{1}{N^d}\bigg)
\end{align*}
for $N\to\infty$ with $p,d$ fixed (cf.\ the proof of Lemma \ref{lem:friedmom}).
It follows that
$$
	\nu_{m}(x^p) = 
	\sum_{i_1,\ldots,i_p=1}^{2d}
	b_{m-1}(g_{i_1}\cdots g_{i_p})
	\,1_{g_{i_1}\cdots g_{i_p}\ne e} =
	\sum_{r=1}^m \sum_{w\in \mathcal{W}_p^r}
	b_{m-1}(w)
$$
for $m\ge 2$, where $\mathcal{W}_p^r =\{w=g_{i_1}\cdots
g_{i_p}: 1\le i_1,\ldots,i_p\le 2d,\pi(w)=r\}$. Thus
$$
	|\nu_m(x^p)| \le
	m p^{2m}
	\max_{r=1,\ldots,m}
	\sum_{w\in \mathcal{W}_p^r}
	|\mathrm{Crit}(w)|,
$$
where we used that $|\mathrm{Crit}(w)|\ge 1$ for $\pi(w)\ne\infty$.
The conclusion
now follows from the second (unnumbered)
theorem in \cite[\S 2.4]{Pud15}. 
\end{proof}


\begin{cor}
\label{cor:pudersupp}
Let $d\ge 2$.
Then $\supp\nu_m \subseteq [-\rho_m,\rho_m]$ for all $0\le m\le d$.
\end{cor}

\begin{proof}
For $m=0,1$ the follows from $\|A_F\|=2\sqrt{2d-1}$ and
Corollary \ref{cor:friedsupp}, while for $m=2,\ldots,d$ this 
follows from Lemmas \ref{lem:puder} and \ref{lem:distsupp}.
\end{proof}

We can now complete the proof of Theorem \ref{thm:puder}.

\begin{proof}[Proof of Theorem \ref{thm:puder}]
Fix $m_*\le m\le d-1$. Let $\chi$ be the test function provided by 
Lemma \ref{lem:test} with $m\leftarrow 4(m+1)$, $K\leftarrow 2d$,
$\rho\leftarrow\rho_m$. Then $\nu_i(\chi)=0$ for $i\le m$ by
Corollary~\ref{cor:pudersupp}. Thus Theorem 
\ref{thm:smmasterperm} with $h\leftarrow\chi$, 
$m\leftarrow m+1$ and Lemma \ref{lem:test}
yield
$$
	\mathbf{P}[\|A^N\| \ge
        \rho_m+\varepsilon] \le
	\mathbf{P}[\tr \chi(A^N)\ge 1] \lesssim
	\frac{C_d}{N^m}
	\frac{1}{\varepsilon^{4(m+1)}}
	\log
	\bigg(\frac{2e}{\varepsilon}\bigg),
$$
where we chose $\beta_*=1+\log(\frac{2}{\varepsilon})$ (note that 
$\beta_*>1$ as $\varepsilon<\rho_{m+1}-\rho_m\le 2$).

For the lower bound, it is shown in the proof of
\cite[Theorem 2.11]{Fri08} that
$$
	\mathbf{P}[\|A^N\| \ge \alpha'] \ge (1-o(1)) N^{1-m} 
$$
for all $m>m_*$ and $\alpha'<\rho_m$. Choosing
$\alpha'=\rho_{m-1}+\varepsilon$ completes the proof.
\end{proof}

\section{Strong convergence of random permutation matrices}
\label{sec:bc}

The aim of this section is to prove Theorem \ref{thm:strongperm}. Most of 
the proof is identical to that of Theorem \ref{thm:friedman}. The only 
difficulty in the present setting is to adapt the argument of Lemma 
\ref{lem:friedsupp} for bounding $\supp\nu_1$. This is not entirely 
straightforward, because the proof of Lemma \ref{lem:friedsupp} relied on 
an overcounting argument which is not applicable to general polynomials. 
Nonetheless, we will show that a more careful implementation of the idea 
behind the proof can avoid this issue.

Throughout the proof, we fix a polynomial $P$ as in Theorem 
\ref{thm:strongperm}, and define
$$
	X^N := P(\boldsymbol{S}^N,\boldsymbol{S}^{N*}),\qquad\quad
	X_{\rm F} := P(\boldsymbol{s},\boldsymbol{s}^*).
$$
To simplify the proof, we begin by applying a linearization trick of 
Pisier \cite{Pis18} in order to factor $X_{\rm F}$ into polynomials of 
degree one.

\begin{lem}
\label{lem:lin}
For any $\varepsilon>0$, there exist $q\in\mathbb{N}$ and operators
$$
	X_j := \sum_{i=0}^{2d} A_{j,i}\otimes \lambda(g_i),\qquad
	j=1,\ldots,q
$$
with matrix coefficients $A_{j,i}$ of dimensions $D_j\times D_{j+1}$ 
with $D_1=D_{q+1}=D$, so that
$X_{\rm F}=X_1\cdots X_q$ and $\|X_j\|\le (\|X_{\rm 
F}\|+\varepsilon)^{1/q}$ for all $j$.
\end{lem}

\begin{proof}
This is an immediate consequence of \cite[Theorem 1]{Pis18}.
\end{proof}

In the following, we fix $\varepsilon>0$ and a factorization of $X_{\rm 
F}$ as in Lemma \ref{lem:lin}. We will also define $X_j^N$ by replacing
$\lambda(g_i)\leftarrow S_i^N$ in the definition of $X_j$. Note that as 
$g_1,\ldots,g_d$ are free, $X_{\rm F}=X_1\cdots X_q$ implies that
$X^N=X^N_1\cdots X^N_q$.

It will be convenient to extend the indexing of the above operators 
cyclically: we will define $X_j := X_{((j-1)\mathop{\mathrm{mod}}q)+1}$ 
for all $j\in\mathbb{N}$. This implies, for example, that $X_{\rm F}^p = 
X_1X_2\cdots X_{pq}$. We similarly define $A_{j,i} := 
A_{((j-1)\mathop{\mathrm{mod}}q)+1,i}$ for $j\in\mathbb{N}$.

\subsection{The first-order support}

In the present setting, the moments of the linear functionals $\nu_0$ and 
$\nu_1$ in Theorem \ref{thm:masterperm} satisfy
\begin{align*}
	\nu_0(x^p) &= \lim_{N\to\infty} \EE[\ntrND (X^N)^p]
	= \tau(X_F^p), \\
	\nu_1(x^p) &= \lim_{N\to\infty} N\big(
	\EE[\ntrND (X^N)^p] - \tau(X_F^p)\big).
\end{align*}
We begin by writing an expression for $\nu_1(x^p)$.

\begin{lem}
\label{lem:safmom}
For every $p\in\mathbb{N}$, we have
$$
        \nu_1(x^p) = - \tau(X_{\rm F}^p) +
        \sum_{k=2}^{pq} (\omega(k)-1)
	\sum_{v\in\mathbf{F}_d^{\rm np}}
        \sum_{i_1,\ldots,i_{pq}=0}^{2d}
	a_{i_1,\ldots,i_{pq}}\,
        1_{g_{i_1}\cdots g_{i_{pq}}=v^k},
$$
where we define
$$
	a_{i_1,\ldots,i_{pq}} := 
	\ntrD\big(A_{1,i_1}A_{2,i_2}\cdots A_{pq,i_{pq}}\big).
$$
\end{lem}

\begin{proof}
The proof is identical to that of Lemma \ref{lem:friedmom}.
\end{proof}

We would like to repeat the proof of Lemma \ref{lem:friedsupp} in the 
present setting. Recall that the idea of the proof is to write 
$v^k=gwww^{k-2}g^{-1}$ where $w$ is cyclically reduced, and then to reason 
that any word $g_{i_1}\cdots g_{i_t}$ that reduces to $v^k$ is a 
concatenation of words $g_{i_1}\cdots g_{i_{t_1}}=g$, $g_{i_{t_1+1}}\cdots 
g_{i_{t_2}}=w$, $g_{i_{t_2+1}}\cdots g_{i_{t_3}}=w$, etc. We can therefore 
sum over all such concatenations in the expression for $\nu_1$.

The problem with this argument is that the above decomposition is not 
unique: for example, $g_1 g_2 g_2^{-1} g_1 = g_1^2$ can be decomposed in 
two different ways $(g_1,g_2g_2^{-1}g_1)$ or $(g_1g_2g_2^{-1},g_1)$. Thus 
when we sum over all possible ways of generating such concatenations, we 
are overcounting the number of words that reduce to $v^k$. Unlike in Lemma 
\ref{lem:friedsupp}, the coefficients $a_{i_1,\ldots,i_{pq}}$ can have 
both positive and negative signs and thus we cannot upper bound the 
moments by overcounting.

We presently show how a more careful analysis can avoid overcounting.
We begin by introducing the following basic notion.

\begin{defn}
\label{defn:triangle}
For any $v\in\mathbf{F}_d$ and $0\le i_1,\ldots,i_k\le 2d$, we write
$g_{i_1}\cdots g_{i_k}\triangleq v$ if
$g_{i_1}\cdots g_{i_k}=v$ and $g_{i_l}\cdots g_{i_k}\ne v$ for
$1<l\le k$.
\end{defn}

To interpret this notion, one should think of any word $g_{i_1}\cdots 
g_{i_k}$ as defining a walk in the Cayley graph of $\mathbf{F}_d$ when 
read from right to left, starting at the identity. Then $g_{i_1}\cdots 
g_{i_k}\triangleq v$ states that the walk defined by $g_{i_1}\cdots 
g_{i_k}$ reaches $v$ \emph{for the first time} at its endpoint. 
Just as we can write
$$
	1_{g_{i_1}\cdots g_{i_k}=v}=\langle 
	\delta_v,\lambda(g_{i_1})\cdots\lambda(g_{i_k})\delta_e\rangle,
$$
the equivalence notion of Definition \ref{defn:triangle} may also be 
expressed in terms of matrix elements of the left-regular 
representation.

\begin{lem}
\label{lem:triangle}
For any $0\le i_1,\ldots,i_k\le 2d$ and $v\in\mathbf{F}_d$, we have
\begin{multline*}
        1_{g_{i_1}\cdots g_{i_k}\triangleq v} =
        \langle 
	\delta_v,\lambda(g_{i_1})\cdots\lambda(g_{i_k})\delta_e\rangle
\\
        -
        \sum_{s=1}^{k-1}
        \langle \delta_e,\lambda(g_{i_1})Q
        \lambda(g_{i_2})\cdots Q
        \lambda(g_{i_s})\delta_e\rangle
        \langle \delta_v,
        \lambda(g_{i_{s+1}})\cdots
        \lambda(g_{i_k})\delta_e\rangle,
\end{multline*}
where we defined $Q := \id - \delta_e\delta_e^*$ (that is,
$Q$ is the orthogonal projection onto $\{\delta_e\}^\perp$).
\end{lem}

\begin{proof}
One may simply note that term $s$ in the sum is the indicator
of the event that $g_{i_1}\cdots g_{i_k}=v$, 
$g_{i_{s+1}}\cdots g_{i_k}=v$, and
$g_{i_t}\cdots g_{i_k}\ne v$ for $1<t<s+1$.
\end{proof}

We now use this identity to express the kind of sum that appears in Lemma 
\ref{lem:safmom} exactly (without overcounting) in terms 
of matrix elements of the operators $X_j$.

\begin{lem}
\label{lem:noovercount}
Fix $v_1,\ldots,v_\ell\ne e\in\mathbf{F}_d$ $(\ell\ge 2)$ so that
$v_1\cdots v_\ell$ is reduced. Define
\begin{align*}
	X_{(t,s)} &:= X_t\cdots X_s, \\
	X_{[t,s]} &:= -X_t(\id_{D_{t+1}}\otimes Q)X_{t+1}\cdots
	(\id_{D_s}\otimes Q)X_s
\end{align*}
for $t\le s$ and $X_{(t,s)}=X_{[t,s]}:=\id$ for $t>s$. Then
\begin{multline*}
        \sum_{i_1,\ldots,i_{pq}=0}^{2d}
	a_{i_1,\ldots,i_{pq}}\,
        1_{g_{i_1}\cdots g_{i_{pq}}=v_1\cdots v_\ell}= \\
	\sum_{x_2,y_2,\ldots,x_\ell,y_\ell,x_{\ell+1}}
	\sum_{1<t_2\le s_2<\cdots<t_\ell\le s_\ell\le pq}
	\langle e_{x_{\ell+1}}\otimes\delta_{v_1},X_{(1,t_2-1)}
	\,e_{x_2}\otimes\delta_e\rangle 
\times\mbox{}
\\
	\prod_{r=2}^{\ell}
	\langle e_{x_r}\otimes
	\delta_e,X_{[t_r,s_r-1]}\,e_{y_r}\otimes\delta_e\rangle
	\langle e_{y_r}\otimes\delta_{v_r},X_{(s_r,t_{r+1}-1)}
	\,e_{x_{r+1}}\otimes\delta_e\rangle
\end{multline*}
where we let $t_{\ell+1}:=pq+1$.
\end{lem}

\begin{proof}
Suppose that $g_{i_1}\cdots g_{i_{pq}}=v_1\cdots v_\ell$. As we assumed 
that
$v_1\cdots v_\ell$ is reduced, there are unique indices
$1<t_2<\cdots<t_\ell\le pq$ so that
$$
	g_{i_1}\cdots g_{i_{t_2-1}}=v_1,\qquad
	g_{i_{t_r}}\cdots g_{i_{t_{r+1}-1}}\triangleq v_r
	\quad\text{for }2\le r\le \ell.
$$
Thus we can write
$$
        1_{g_{i_1}\cdots g_{i_{pq}}=v_1\cdots v_\ell} =
        \sum_{1<t_2<\cdots<t_\ell\le pq}
        1_{g_{i_1}\cdots g_{i_{t_2-1}}=v_1}
        \prod_{r=2}^\ell
        1_{g_{i_{t_r}}\cdots g_{i_{t_{r+1}-1}}\triangleq v_r}.
$$
The conclusion follows by applying Lemma \ref{lem:triangle}
to the indicators in this identity and summing over the indices 
$i_1,\ldots,i_{pq}$.
\end{proof}

With this identity in hand, we can proceed
exactly as in Lemma \ref{lem:friedsupp}. In the following lemma, recall
that we chose an arbitrary $\varepsilon>0$ in Lemma \ref{lem:lin}.

\begin{lem}
\label{lem:saffriedman}
For every $k\ge 2$, we have 
$$
	\Bigg|
	\sum_{v\in\mathbf{F}_d^{\rm np}}
        \sum_{i_1,\ldots,i_{pq}=0}^{2d}
	a_{i_1,\ldots,i_{pq}}\,
        1_{g_{i_1}\cdots g_{i_{pq}}=v^k}
	\Bigg|
	\le
	2D_{\rm max}^9(pq)^8 (\|X_{\rm F}\|+\varepsilon)^p,
$$
where $D_{\rm max}=\max_j D_j$.
\end{lem}

\begin{proof}
Suppose first that $k\ge 3$. Let $\mathcal{F}$ be the set of $v\in 
\mathbf{F}_d^{\rm np}$ that are cyclically reduced, and let 
$\mathcal{F}':=\mathbf{F}_d^{\rm np}\backslash \mathcal{F}$. Then every 
$v\in\mathcal{F}'$ can be uniquely expressed as $v=gwg^{-1}$ with 
$w\in\mathcal{F}$, $g\in\mathbf{F}_d$ so that $gwg^{-1}$ is reduced.
Applying Lemma \ref{lem:noovercount} with $\ell=5$ and
$v_1\leftarrow g$, $v_2\leftarrow w$, $v_3\leftarrow w$, $v_4\leftarrow 
w^{k-2}$, $v_5\leftarrow g^{-1}$ yields
\begin{multline*}
	\Bigg|
	\sum_{v\in\mathcal{F}'}
        \sum_{i_1,\ldots,i_{pq}=0}^{2d}
	a_{i_1,\ldots,i_{pq}}\,
        1_{g_{i_1}\cdots g_{i_{pq}}=v^k}
	\Bigg|	
	\le \\
	\sum_{1<t_2\le s_2<\cdots<t_5\le s_5\le pq}
	\|X_{[t_2,s_2-1]}\|
	\|X_{[t_3,s_3-1]}\|
	\|X_{[t_4,s_4-1]}\|
	\|X_{[t_5,s_5-1]}\|
	\|X_{(s_4,t_5-1)}\| 
 \times \mbox{}\\
	\sum_{x_2,y_2,\ldots,x_5,y_5,x_6}
	\sum_{w\in\mathbf{F}_d}
	|\langle e_{y_2}\otimes\delta_{w},X_{(s_2,t_3-1)}
	\, e_{x_3}\otimes\delta_e\rangle
	\langle e_{y_3}\otimes\delta_{w},
	X_{(s_3,t_4-1)}\, e_{x_4}\otimes\delta_e\rangle|
\times\mbox{}\\
	\sum_{g\in\mathbf{F}_d}
	|\langle e_{x_6}\otimes\delta_{g},X_{(1,t_2-1)}
	\,e_{x_2}\otimes\delta_e\rangle
	\langle	e_{y_5}\otimes\delta_{g^{-1}},X_{(s_5,pq)}
	\,e_{x_6}\otimes\delta_e\rangle|.
\end{multline*}
Using that 
$\|X_{(t,s-1)}\|\le (\|X_{\rm F}\|+\varepsilon)^{(s-t)/q}$
and
$\|X_{[t,s-1]}\|\le (\|X_{\rm F}\|+\varepsilon)^{(s-t)/q}$
by Lemma \ref{lem:lin} and $\|Q\|=1$, and applying Cauchy--Schwarz 
as in the proof of Lemma~\ref{lem:friedsupp}, we can upper bound
the right-hand side by $D_{\rm max}^9 (pq)^8(\|X_{\rm 
F}\|+\varepsilon)^p$.

If we sum instead over $v\in\mathcal{F}$ on the left-hand side, we can 
bound in exactly the same manner by applying Lemma
\ref{lem:noovercount} with $\ell=3$ and $v_1\leftarrow w$, $v_2\leftarrow 
w$, $v_3\leftarrow w^{k-2}$. Thus the proof is complete for the case
$k\ge 3$.

The proof for $k=2$ is identical, except that we now omit the $w^{k-2}$ 
terms.
\end{proof}


\begin{cor}
\label{cor:safsupp}
$\supp \nu_1 \subseteq [-\|X_{\rm F}\|,\|X_{\rm F}\|]$.
\end{cor}

\begin{proof}
Lemmas \ref{lem:safmom} and \ref{lem:saffriedman} imply that
$|\nu_1(x^p)| \le \|X_{\rm F}\|^p + 2 D_{\rm max}^9 (pq)^{10} (\|X_{\rm 
F}\|+\varepsilon)^p$. The conclusion follows by first applying Lemma 
\ref{lem:distsupp}, and then letting $\varepsilon\to 0$.
\end{proof}

\subsection{Proof of Theorem \ref{thm:strongperm}}

The rest of the proof contains no new ideas.

\begin{proof}[Proof of Theorem \ref{thm:strongperm}]
Let $\chi$ be the test function provided by Lemma \ref{lem:test} with 
$m=8$, $K=\|P\|_{\mathrm{M}_d(\mathbb{C})\otimes C^*(\mathbf{F}_d)}$, and 
$\rho=\|X_{\rm F}\|$. Then $\nu_0(\chi)=\nu_1(\chi)=0$ by Corollary 
\ref{cor:safsupp}.
Thus Theorem
\ref{thm:smmasterperm} with $h\leftarrow\chi$,
$m\leftarrow 2$ and Lemma \ref{lem:test} yield
$$
	\mathbf{P}[\|X^N\|\ge\|X_{\rm F}\|+\varepsilon]
	\le
	\mathbf{P}[\tr \chi(X^N) \ge 1]
	\lesssim 
	\frac{D}{N}
	\bigg(\frac{Kq_0\log d}{\varepsilon}\bigg)^{8}
	\log
	\bigg(\frac{eK}{\varepsilon}\bigg)
$$
for all $\varepsilon < K-\|X_{\rm F}\|$,
where we chose $\beta_*=1+\log(\frac{K}{\varepsilon})$.
\end{proof}

\section{Stable representations of the symmetric group}
\label{sec:stable}

The aim of this section is to prove Theorem \ref{thm:strongstable}. The 
main issue here is to extend the basic properties of random permutations 
matrices of section \ref{sec:words} to general stable representations. 
With analogues of these properties in place, the remainder of the
proof is nearly the same as that of Theorem \ref{thm:strongperm}. 

Throughout this section, we fix a stable representation as in section 
\ref{sec:stabmain} and adopt the notations introduced there. In addition, 
we will denote by $\beta$ the degree of the polynomial 
$\varphi(x_1,x_2^2,\ldots,x_r^r)$ and by $\kappa:= \sup_{N\ge N_0}\frac{D_N}{N^\alpha}$.

\subsection{Basic properties}

We will deduce the requisite properties from the much more precise results 
of \cite{HP23}. (In fact, the only facts that will be needed here are 
relatively elementary, as is explained in \cite[Remark 31]{PL10} and 
\cite[\S 1.5.1]{HP23}.)

Let $w\in\mathbf{W}_d$ be a word. Then clearly
$$
	w(\boldsymbol{\Pi}^N) = \pi_N(w(\boldsymbol{\sigma}))
$$
where $\boldsymbol{\sigma}=(\sigma_1,\ldots,\sigma_d)$ are i.i.d.\ 
uniformly distributed elements of $\mathbf{S}_N$.
The distribution of $w(\boldsymbol{\sigma})$ defines a probability measure
$\mathbf{P}_{w,N}$ on $\mathbf{S}_N$, called the word measure.

\begin{lem}
\label{lem:stabrational}
Let $w\in\mathbf{W}_d$ be any word of length at most $q$ that does
not reduce to the identity, and let $N\ge \max\{\beta q, N_0\}$. Then we have
$$
        \mathbf{E}[\tr w(\boldsymbol{\Pi}^N)] =
	\frac{f_w(\tfrac{1}{N})}{g_q(\tfrac{1}{N})},
$$
where
$$
	g_q(x) :=
	(1-x)^{d_1}(1-2x)^{d_2}\cdots 
	(1-(\beta q-1)x)^{d_{\beta q-1}}
$$
with $d_j := \min\big(d,\lfloor \frac{\beta q}{j+1}\rfloor\big)$, 
and $f_w,g_q$ are polynomials of degree at most $\beta q(1+\log d)$.
\end{lem}

\begin{proof}
We may assume without loss of generality that $w$ is cyclically reduced
with length $1\le\ell\le q$, as $\tr w(\boldsymbol{\Pi}^N)$ is invariant 
under cyclic reduction.

The definition of the stable representation implies that
\begin{equation}
\label{eq:stablehp}
	\mathbf{E}[\tr w(\boldsymbol{\Pi}^N)] =
	\mathbf{E}_{w,N}[\varphi(\xi_1,\ldots,\xi_r)].
\end{equation}
Let $\xi_1^{\alpha_1}\cdots \xi_r^{\alpha_r}$ be a monomial 
of $\varphi$ of degree at least one. Then for $N$ sufficiently large,
\cite[Example 3.6, Definition 6.6, and Proposition 6.8]{HP23} imply that
\begin{equation}
\label{eq:stabmono}
	\mathbf{E}_{w,N}[\xi_1^{\alpha_1}\cdots \xi_r^{\alpha_r}]
	=
        \sum_\Gamma
        \frac{
        (\frac{1}{N})^{e_\Gamma-v_\Gamma}
        \prod_{i=1}^{v_\Gamma-1}(1-\frac{i}{N})}
        {\prod_{j=1}^d \prod_{i=1}^{e_\Gamma^j-1} (1-\frac{i}{N})},
\end{equation}
where the sum is over a certain collection of quotients $\Gamma$ of the 
graph consisting of $\alpha_1$ disjoint cycles of length $\ell$, 
$\alpha_2$ disjoint cycles of length $2\ell$, etc., whose edges are 
colored by $\{1,\ldots,d\}$. Here we denote by
$v_\Gamma$ the number of vertices of $\Gamma$, 
by $e_\Gamma^j$ the number of edges of $\Gamma$ colored by $j$,
and by $e_\Gamma=e_\Gamma^1+\cdots+e_\Gamma^d$.

It is immediate from the above construction that $v_\Gamma \le 
\ell\sum_{i=1}^r i\alpha_i \le \beta q$, and analogously $e_\Gamma \le 
\beta q$. The validity of \eqref{eq:stabmono} for all $N\ge \beta q\ge 
\max_j e_\Gamma^j$ can be read off from the proof of \cite[Proposition 
6.8]{HP23}. Moreover, we must have $e_\Gamma -v_\Gamma \ge 0$ for all 
$\Gamma$ that appear in the sum, as 
$\mathbf{E}_{w,N}[\xi_1^{\alpha_1}\cdots \xi_r^{\alpha_r}]=O(1)$ by 
\cite[Theorem 1.3]{HP23} and the following discussion.
Thus \eqref{eq:stabmono} is a rational function of $\frac{1}{N}$.
The remainder of the proof proceeds exactly as in the proof of
Lemma \ref{lem:wordrational}.
\end{proof}

We now describe the lowest order asymptotics.

\begin{lem}
\label{lem:stablenica}
Fix a word $w\in\mathbf{W}_d$ that does not reduce to the identity, and express 
its reduction as $w(g_1,\ldots,g_d)=v^k$ with
$v\in\mathbf{F}_d^{\rm np}$ and $k\ge 1$. Then
$$
        \lim_{N\to\infty}
        \mathbf{E}[\tr w(\boldsymbol{\Pi}^N)] = \varpi(k)
$$
with $\varpi(1)=0$ and $|\varpi(k)|\le c k^\beta$ for
$k\ge 2$. Here $c$ is a constant that depends on $\varphi$.
\end{lem}

\begin{proof}
Recall that $\varphi$ is a sum of character polynomials of irreducible 
representations, and that we assumed that $\pi_N$ does not contain the 
trivial representation. The case $k=1$ therefore follows from Lemma 
\ref{lem:nica} for the standard representation and from \cite[Corollary 
1.7]{HP23} for all other irreducible components of $\varphi$.

For $k\ge 2$, let $\xi_1^{\alpha_1}\cdots\xi_r^{\alpha_r}$ be a monomial 
of $\varphi$. Applying \cite[Theorem 1.3]{HP23} and the following 
discussion as well as \cite[Remark 7.3]{HP23} yields
$$
	\lim_{N\to\infty} 
	\mathbf{E}_{w,N}[\xi_1^{\alpha_1}\cdots\xi_r^{\alpha_r}] =
	\lim_{N\to\infty}
	\mathbf{E}_{v,N}[\xi_k^{\alpha_1}\cdots\xi_{rk}^{\alpha_r}]
	=
	\sum_{\mathcal{P}\in\mathrm{Partitions}(S)}
	\prod_{A\in\mathcal{P}}
	\sum_{j|\mathrm{gcd}(A)} j^{|A|-1},
$$
where $S$ is the multiset that contains $jk$ with 
multiplicity $\alpha_j$ for $j=1,\ldots,r$.
By crudely estimating $j\le \mathrm{gcd}(A)\le rk$, the right-hand side
is bounded by $B_{|S|}(rk)^{|S|}$ where $B_n\le n!$ denotes
the number of partitions of a set with $n$ elements. 
The conclusion follows as $|S|=\sum_{j=1}^r\alpha_j\le\beta$ and using 
\eqref{eq:stablehp}.
\end{proof}

Note that an immediate consequence of Lemma \ref{lem:stablenica} is that 
the weak convergence as in Corollary \ref{cor:waf} remains valid in the 
present setting.

\subsection{Proof of Theorem \ref{thm:strongstable}}

Fix a self-adjoint noncommutative polynomial $P\in 
\mathrm{M}_D(\mathbb{C}) \otimes 
\mathbb{C}\langle\boldsymbol{s},\boldsymbol{s}^*\rangle$ of degree $q_0$, 
and let $h\in\mathcal{P}_q$. Then 
$({\tr}\otimes\mathrm{id})(h(P(\boldsymbol{\Pi}^N,\boldsymbol{\Pi}^{N*})))$ 
is a linear combination of words $w(\boldsymbol{\Pi}^N)$ of length at most 
$qq_0$. Summing separately over words that do and do not reduce to the 
identity yields
$$
	\mathbf{E}[\tr h(P(\boldsymbol{\Pi}^N,\boldsymbol{\Pi}^{N*}))] =
	\frac{\tilde f_h(\frac{1}{N})}{g_{qq_0}(\frac{1}{N})} +
	D_N\, (\tr\otimes\tau)(h(P(\boldsymbol{s},\boldsymbol{s}^*)))
$$
by Lemma \ref{lem:stabrational},
where $\tilde f_h$ is a polynomial of degree at most $\beta qq_0(1+\log d)$.
Thus
$$
	\mathbf{E}[\ntrNalp h(P(\boldsymbol{\Pi}^N,\boldsymbol{\Pi}^{N*}))] 
	= \Psi_h(\tfrac{1}{N}) =
	\frac{f_h(\frac{1}{N})}{g_{qq_0}(\frac{1}{N})}
$$
where $f_h$ is the polynomial of degree at most $\beta qq_0(1+\log d)+\alpha
\le 2\beta qq_0(1+\log d)$.
Here we used that $\frac{D_N}{N^\alpha}$ is a polynomial of $\frac{1}{N}$ 
of degree $\alpha$. 

We can now repeat the proofs in sections 
\ref{sec:masterI}--\ref{sec:masterII} nearly verbatim with the 
replacements $q\leftarrow 
2\beta N_0^{1/2}q$ (here we multiply by $N_0^{1/2}$ to 
ensure $M\ge N_0$ in Lemma~\ref{lem:psihbd}) and 
$|\Psi_h(\frac{1}{N})|\le \frac{DD_N}{N^\alpha} \|h\|_{C^0[-K,K]}\le 
\kappa D \|h\|_{C^0[-K,K]}$ to conclude the following.

\begin{lem}
\label{lem:smmasterstable}
Let $d\ge 2$. Fix $P$ as in Theorem \ref{thm:strongstable} and
$K=\|P\|_{\mathrm{M}_D(\mathbb{C})\otimes C^*(\mathbf{F}_d)}$.
Then there exist compactly supported
distributions $\nu_i$ such that
$$
        \bigg|\EE[
        \ntrNalp h(P(\boldsymbol{\Pi}^N,\boldsymbol{\Pi}^{N*}))]
        - \sum_{i=0}^{\alpha} \frac{\nu_i(h)}{N^i}
        \bigg|
        \le
        \frac{CD\,(q_0\log d)^{4(\alpha+1)}}{N^{\alpha+1}}
        \,\beta_*\|f^{(4(\alpha+1)+1)}\|_{L^\beta[0,2\pi]}
$$
for all $N\ge N_0$, $\beta>1$, and $h\in C^\infty(\mathbb{R})$.
Here $f(\theta):=h(K\cos\theta)$, $1/\beta_*:=1-1/\beta$, and $C$ is a
constant that depends on the choice of stable representation.
\end{lem}

Next, we must characterize the supports of the distributions
$\nu_i$.

\begin{lem}
\label{lem:stablesupp}
$\supp \nu_i \subseteq [-\|P(\boldsymbol{s},\boldsymbol{s}^*)\|,
\|P(\boldsymbol{s},\boldsymbol{s}^*)\|]$ for $i=0,\ldots,\alpha$.
\end{lem}

\begin{proof}
Suppose first that $0\le i<\alpha$.
As Lemma \ref{lem:stablenica} implies that
$$
	\mathbf{E}[\ntrNalp h(P(\boldsymbol{\Pi}^N,\boldsymbol{\Pi}^{N*}))] 
=
	\frac{D_N}{N^\alpha}\,
	(\tr\otimes\tau)(h(P(\boldsymbol{s},\boldsymbol{s}^*)))
	+ O\bigg(\frac{1}{N^\alpha}\bigg)
$$
as $N\to\infty$ for every $h\in\mathcal{P}$, it follows that there is a 
constant $c_i$ for every $i<\alpha$ so that 
$\nu_i(h)=c_i\,(\tr\otimes\tau)(h(P(\boldsymbol{s},\boldsymbol{s}^*)))$. 
The claim follows directly.

We now consider $i=\alpha$. Then we can repeat the proof of Lemma 
\ref{lem:safmom}, but using Lemma \ref{lem:stablenica} instead of
Lemma \ref{lem:nica}, to obtain the representation
$$
        \nu_\alpha(x^p) = D_0\tau(P(\boldsymbol{s},\boldsymbol{s}^*)^p) +
        \sum_{k=2}^{pq} \varpi(k)
	\sum_{v\in\mathbf{F}_d^{\rm np}}
        \sum_{i_1,\ldots,i_{pq}=0}^{2d}
	a_{i_1,\ldots,i_{pq}}\,
        1_{g_{i_1}\cdots g_{i_{pq}}=v^k}
$$
where $D_0 := \varphi(0,\ldots,0)$ is the constant term in the polynomial 
$N\mapsto D_N$. We therefore obtain
$|\nu_\alpha(x^p)| \lesssim D_{\rm max}^9
(pq)^{9+\beta}(\|P(\boldsymbol{s},\boldsymbol{s}^*)\|+\varepsilon)^p$
by Lemmas \ref{lem:saffriedman} and \ref{lem:stablenica}, and
the conclusion follows by Lemma \ref{lem:distsupp} and as
$\varepsilon>0$ is arbitrary.
\end{proof}

The rest of the proof follows in the usual manner.

\begin{proof}[Proof of Theorem \ref{thm:strongstable}]
Let $\chi$ be the test function provided by Lemma
\ref{lem:test} with $m\leftarrow 4(\alpha+1)$ and
$\rho\leftarrow \|P(\boldsymbol{s},\boldsymbol{s}^*)\|$. Then 
$\nu_i(\chi)=0$ for all $i\le \alpha$ by
Lemma \ref{lem:stablesupp}. Thus Lemma
\ref{lem:smmasterstable} with $h\leftarrow\chi$, 
and Lemma \ref{lem:test} yield
$$
	\mathbf{P}[\|P(\boldsymbol{\Pi}^N,\boldsymbol{\Pi}^{N*})\| \ge
	\|P(\boldsymbol{s},\boldsymbol{s}^*)\|+\varepsilon] \le
	\frac{CD}{N}
	\bigg(\frac{Kq_0\log d}{\varepsilon}\bigg)^{4(\alpha+1)}
        \log
	\bigg(\frac{eK}{\varepsilon}\bigg),
$$
where we chose $\beta_*=1+\log(\frac{K}{\varepsilon})$.
\end{proof}

\appendix

\section{Lower bounds}
\label{sec:lowerbd}

There is a general principle that any random matrix model that satisfies
$$
	\|P(\boldsymbol{X}^N,\boldsymbol{X}^{N*})\| \le
	\|P(\boldsymbol{x},\boldsymbol{x}^*)\| + o(1)
	\quad\text{with probability }1-o(1)
$$
for every noncommutative polynomial $P$, automatically also 
satisfies
$$
	\|P(\boldsymbol{X}^N,\boldsymbol{X}^{N*})\| \ge
	\|P(\boldsymbol{x},\boldsymbol{x}^*)\| - o(1)
	\quad\text{with probability }1-o(1)
$$
provided that the $C^*$-algebra generated by $\boldsymbol{x}$ has a unique 
faithful trace; this can be shown, for example, along the lines of 
\cite[pp.\ 16--19]{LM23}. Since the latter property holds for the limiting 
model defined in section \ref{sec:prelim} (cf.\ \cite{Pow75}), this 
general principle applies to all the models considered in this paper.

For this reason, we have focused exclusively on strong convergence upper 
bounds in the main part of this paper. On the other hand, such a general 
principle is usually not needed in practice, as in most cases the lower 
bound already follows easily from \emph{weak} convergence alone.
For completness, we presently provide an elementary proof of the lower 
bound in the setting of this paper along these lines.

\begin{thm}
\label{thm:lowerstable}
Let $P\in \mathrm{M}_D(\mathbb{C})
\otimes \mathbb{C}\langle\boldsymbol{s},\boldsymbol{s}^*\rangle$ be any
noncommutative polynomial, and fix any stable representation as in section
\ref{sec:stabmain}. Then
$$
        \|P(\boldsymbol{\Pi}^N,\boldsymbol{\Pi}^{N*})\|\ge
        \|P(\boldsymbol{s},\boldsymbol{s}^*)\|-o(1)
        \quad\text{with probability }1-o(1)
$$
as $N\to\infty$.
\end{thm}

Note that all the models in this paper may be viewed as special cases of 
stable representations, so that Theorem \ref{thm:lowerstable} captures 
random permutation matrices and random regular graphs as well.
The proof of Theorem \ref{thm:lowerstable} could also form the basis for 
quantitative lower bounds, but we do not pursue this here.

We begin by recalling the basic fact that Theorem \ref{thm:lowerstable} 
will follow directly once we establish weak convergence in the following 
form.

\begin{lem}[Weak convergence in probability]
\label{lem:lowerweak}
In the setting of Theorem \ref{thm:lowerstable},
$$
        {\ntrDDN}\big(P(\boldsymbol{\Pi}^N,\boldsymbol{\Pi}^{N*})\big) 
        =
        \big({\ntrD}\otimes\tau)(P(\boldsymbol{s},\boldsymbol{s}^*)\big)
        +o(1)
$$
with probability $1-o(1)$ as $N\to\infty$ for every self-adjoint
$P\in \mathrm{M}_D(\mathbb{C})
\otimes \mathbb{C}\langle\boldsymbol{s},\boldsymbol{s}^*\rangle$.
\end{lem}

Let us first use this result to complete the proof of 
Theorem \ref{thm:lowerstable}.

\begin{proof}[Proof of Theorem \ref{thm:lowerstable}]
Fix the polynomial $P$. Then
$$
        \|P(\boldsymbol{\Pi}^N,\boldsymbol{\Pi}^{N*})\|^{2p}
        \ge
        {\ntrDDN}\big(|P(\boldsymbol{\Pi}^N,\boldsymbol{\Pi}^{N*})|^{2p}\big)
        \ge 
        \big({\ntrD}\otimes\tau)(|P(\boldsymbol{s},\boldsymbol{s}^*)|^{2p}\big)
        -o(1)
$$
with probability $1-o(1)$ as $N\to\infty$ for every $p\in\mathbb{N}$ by
Lemma \ref{lem:lowerweak}. Moreover,
$$
        \big({\ntrD}\otimes\tau)(|P(\boldsymbol{s},\boldsymbol{s}^*)|^{2p}\big)^{1/2p}
        = \|P(\boldsymbol{s},\boldsymbol{s}^*)\|-o(1)
$$
as $p\to\infty$ due to the fact that $\tau$ defines a faithful state on 
the $C^*$-algebra $C^*_{\rm red}(\mathbf{F}_d)$ that is
generated by $\boldsymbol{s}$ \cite[Proposition 3.12]{NS06}. The 
conclusion follows.
\end{proof}

The only subtlety in implementing this idea is that the form of weak 
convergence that follows from Lemma \ref{lem:stablenica} only 
yields convergence in expectation
\begin{equation}
\label{eq:lowerweakex}
        \mathbf{E}\big[
        {\ntrDDN} P(\boldsymbol{\Pi}^N,\boldsymbol{\Pi}^{N*})
        \big]
        =
        \big({\ntrD}\otimes\tau)(P(\boldsymbol{s},\boldsymbol{s}^*)\big)
        +o(1),
\end{equation}
as opposed to convergence in probability as in 
Lemma \ref{lem:lowerweak}. 
We must therefore still show concentration
around the mean. The latter can be deduced from \eqref{eq:lowerweakex}
using a minor additional argument. Note that 
\eqref{eq:lowerweakex} does not require $P$ to be self-adjoint.

\begin{proof}[Proof of Lemma \ref{lem:lowerweak}]
Since every stable representation is the direct sum 
of irreducible stable represenations (cf.\ section \ref{sec:stabrep}),
we can assume in the rest of the proof that 
$\pi_N$ is irreducible. Morover, given any 
$P\in \mathrm{M}_D(\mathbb{C})
\otimes \mathbb{C}\langle\boldsymbol{s},\boldsymbol{s}^*\rangle$,
$$
	\ntrDDN P(\boldsymbol{\Pi}^N,\boldsymbol{\Pi}^{N*}) =
	\frac{1}{D}
	\sum_{i=1}^D \ntrDN P_{ii}(\boldsymbol{\Pi}^N,\boldsymbol{\Pi}^{N*})
$$
where $P_{ii}\in \mathbb{C}\langle\boldsymbol{s},\boldsymbol{s}^*\rangle$ 
are the diagonal elements of $P$ with respect to the matrix coefficients. 
We can therefore assume in the rest of the proof that $D=1$.

Let $\Pi^N_0=\pi_N(\sigma_0)$, where
$\sigma_0$ is a uniformly distributed random element of $\mathbf{S}_N$
that is independent of $\sigma_1,\ldots,\sigma_d$.
Moreover, let $s_0=\lambda(g_0)$ where $g_0,\ldots,g_d$ are free 
generators of $\mathbf{F}_{d+1}$.
It follows from Schur's lemma that
$$
	\mathbf{E}[\Pi^N_0 M \Pi^{N*}_0] =
	\ntrDN(M) \id
$$
for every $M\in\mathrm{M}_{D_N}(\mathbb{C})$. Therefore
$$
	\mathbf{E}\big[\ntrDN \Pi^N_0
	P(\boldsymbol{\Pi}^N,\boldsymbol{\Pi}^{N*})
	\Pi^{N*}_0 P(\boldsymbol{\Pi}^N,\boldsymbol{\Pi}^{N*})\big] =
	\mathbf{E}\big[
	(\ntrDN P(\boldsymbol{\Pi}^N,\boldsymbol{\Pi}^{N*}))^2\big].
$$
Moreover, it is readily verified using the definitions in
section \ref{sec:prelim} that
$$
	\tau\big(s_0 P(\boldsymbol{s},\boldsymbol{s}^*) s^*_0
	P(\boldsymbol{s},\boldsymbol{s}^*)\big) =
	\big(\tau(P(\boldsymbol{s},\boldsymbol{s}^*))\big)^2.
$$
Since the objects inside the traces on the left-hand side are 
noncommutative polynomials in $d+1$ variables and as $P$ is self-adjoint, 
applying \eqref{eq:lowerweakex} yields
$$
	\lim_{N\to\infty}
	\mathrm{Var}\big(\ntrDN P(\boldsymbol{\Pi}^N,\boldsymbol{\Pi}^{N*})
	\big) 
	=0.
$$
Given that convergence in expectation follows from \eqref{eq:lowerweakex} 
and that the variance converges to zero, convergence in probability 
follows immediately.
\end{proof}

\subsection*{Acknowledgments}

\!\!We are grateful to Charles Bordenave, Beno{\^i}t Collins, Michael 
Magee, Doron Puder, and Mikael de la Salle for very helpful discussions, 
and to Nick Cook, Srivatsav Kunnawalkam Elayavalli, Joel Friedman, 
Sidhanth Mohanty, Peter Sarnak, Nikhil Srivastava, Pierre Youssef, and 
Yizhe Zhu for helpful comments. We are especially indebted to Michael 
Magee for suggesting the application to stable representations, and for 
patiently explaining what they are. RvH thanks Peter Sarnak for organizing 
a stimulating seminar on strong convergence and related topics at 
Princeton in Fall 2023. Finally, we are grateful to the referee for 
suggestions that helped us improve the presentation of the paper.

CFC was supported by a Simons-CIQC postdoctoral fellowship through NSF 
grant QLCI-2016245. JGV was supported by NSF grant FRG-1952777, Caltech 
IST, and a Baer--Weiss postdoctoral fellowship. JAT was supported by the 
Caltech Carver Mead New Adventures Fund, NSF grant FRG-1952777, and ONR 
grants BRC-N00014-18-1-2363 and N00014-24-1-2223. RvH was supported by NSF 
grants DMS-2054565 and DMS-2347954.

\bibliographystyle{abbrv}
\bibliography{ref}

\end{document}